\documentclass[11pt]{article}
\usepackage{psfrag,cite,amsmath,graphicx,epsfig,latexsym,subfigure,color}
\usepackage[scriptsize,bf]{caption}
\usepackage{color}
\usepackage{amssymb}
\DeclareMathOperator{\argmin}{\mbox{argmin}}
\def\bn{\hfill \\ \smallskip\noindent}
\def\argmin{\mathop{\rm argmin}}

\def\vx{x}

\def\prox{\mbox{prox}}

\def\dist{\mbox{dist\,}}
\renewcommand{\baselinestretch}{1.2}   %line space adjusted here
\newcommand{\beq}{\begin{equation}}
\newcommand{\eeq}{\end{equation}}
\newcommand{\st}{{\rm s.t.}}

\newcommand{\cI}{\mathcal{I}}

\newcommand{\bfPsi}{{\mbox{\boldmath $\Psi$}}}

\begin{document}

\bigskip
\def\theequation {\thesection.\arabic{equation}}
\def\pn {\par\smallskip\noindent}
\def \bn {\hfill \\ \smallskip\noindent}
\newcommand{\fs}{f_1,\ldots,f_s}
\newcommand{\f}{\vec{f}}
\newcommand{\hx}{\hat{x}}
\newcommand{\hy}{\hat{y}}
\newcommand{\barhx}{\bar{{x}}}
\newcommand{\vecx}{x_1,\ldots,x_m}
\newcommand{\xoy}{x\rightarrow y}
\newcommand{\barx}{{\bar x}}
\newcommand{\bary}{{\bar y}}
\newtheorem{theorem}{Theorem}[section]
\newtheorem{lemma}{Lemma}[section]
\newtheorem{corollary}{Corollary}[section]
\newtheorem{proposition}{Proposition}[section]
\newtheorem{definition}{Definition}[section]
\newtheorem{claim}{Claim}[section]
\newtheorem{remark}{Remark}[section]

\newcommand{\newsection}{\setcounter{equation}{0}\section}

\def\br{\break}
\def\smskip{\par\vskip 5 pt}
\def\proof{\bn {\bf Proof.} }
\def\QED{\hfill{\bf Q.E.D.}\smskip}
\def\qed{\quad{\bf q.e.d.}\smskip}

\newcommand{\cM}{\mathcal{M}}
\newcommand{\cJ}{\mathcal{J}}
\newcommand{\cT}{\mathcal{T}}
\newcommand{\bx}{\boldsymbol{x}}
\newcommand{\bp}{\boldsymbol{p}}
\newcommand{\bz}{\boldsymbol{z}}

%% PUT YOUR TITLE PAGE INFORMATION HERE %%%
\begin{titlepage}
\title{\bf A Block Successive Upper Bound Minimization Method of Multipliers for Linearly Constrained Convex Optimization}
\author{Mingyi Hong\thanks{University of Minnesota, Minneapolis, USA, email: \texttt{\{mhong, meisam, luozq\}@umn.edu}.}, \quad Tsung-Hui Chang\thanks{National Taiwan University of Science and Technology, Taiwan, R.O.C., email: \texttt{tsunghui.chang@ieee.org}}, \quad Xiangfeng Wang \thanks{Nanjing University, Nanjing, P.R.~ China, email: \texttt{wangxf@smail.nju.edu.cn}}, \\Meisam Razaviyayn$^{*}$, \quad Shiqian Ma\thanks{The Chinese University of Hong Kong, Hong Kong, P.R.~ China, email: \texttt{sqma@se.cuhk.edu.hk}}, \quad Zhi-Quan Luo$^{*}$\let\thefootnote\relax\footnote{The research of M.~Hong, M.~Razaviyayn and Z.-Q.~Luo is supported in part by the AFOSR, grant number  FA9550-12-1-0340. The research of S.~ Ma is supported by the Hong Kong Research Grants Council (RGC) Early Career Scheme (ECS) (Project ID: CUHK 439513). T.-H.~ Chang is supported by the  National Science Council, Taiwan (R.O.C.), under grant NSC 102-2221-E-011-005-MY3.}
}
%\thanks{$^1 $ University of Minnesota, Minneapolis, USA;
%$^2$ National Taiwan University of Science and Technology, Taiwan, R.O.C.;
% $^3$ Nanjing University, Nanjing, P. R. China;
% $^4$ The Chinese University of Hong Kong, Hong Kong, P. R. China}}
\maketitle
\date{}
\vspace*{-1cm}

\begin{abstract}
%In this paper, we consider a (possibly nonsmooth) convex optimization problem with multiple blocks of variables and a linear constraint coupling all the variables. We propose an algorithm called {\it block coordinate descent method of multipliers} (BSUM-M) to solve this family of problems. The BSUM-M is a primal-dual type of algorithm. It integrates the traditional block coordinate decent (BCD) algorithm and the alternating direction method of multipliers (ADMM), in which it optimizes the (approximate ) augmented Lagrangian of the original problem one block variable each time, followed by a gradient update for the dual variable. The BSUM-M is more general than both the BCD and the ADMM, due to its ability to simultaneously deal with coupled objective function and linearly coupled constraints.
%
% The focus of this paper is to analyze the convergence for this family of algorithm. Under certain regularity conditions, and when the order for which the block variables are either updated in a deterministic way or in a random fashion, we show that the BSUM-M converges, under either a diminishing stepsize rule or a small enough stepsize for the dual update. We then analyze an important special case of the (randomized) BSUM-M algorithm, the well-known (randomized) BCD method, which handles a subset of our considered problem without linearly coupled constraint. We show that under certain regularity conditions, the randomized BCD method in fact also converges linearly.
%
 Consider the problem of minimizing the sum of a smooth convex function and a separable nonsmooth convex function subject to linear coupling constraints. Problems of this form arise in many contemporary applications including signal processing, wireless networking and smart grid provisioning. Motivated by the huge size of these applications, we propose a new class of first order primal-dual algorithms called the
 %{\it block coordinate descent method of multipliers} (BCDMM)
 {\it block successive upper-bound minimization method of multipliers} (BSUM-M)
 to solve this family of problems. The BSUM-M updates the primal variable blocks successively by minimizing locally tight upper-bounds of the augmented Lagrangian of the original problem, followed by a gradient type update for the dual variable in closed form. We show that under certain regularity conditions, and when the primal block variables are updated in either a deterministic or a random fashion, the BSUM-M converges to the set of optimal solutions. Moreover, in the absence of linear constraints, we show that the BSUM-M, which reduces to the \emph{block successive upper-bound minimization } (BSUM \cite{Razaviyayn12SUM}) method, is capable of linear convergence without strong convexity. %Finally, we illustrate the effectiveness of the BSUM-M using two large scale applications from basis pursuit and smart grid.
\end{abstract}

%\vspace*{\fill}
\pn
\noindent {\bf AMS(MOS) Subject Classifications:}  49, 90.
\end{titlepage}

%%SO MUCH FOR THE TITLE PAGE

\newsection{\bf Introduction} \label{sub:intro}

Consider the problem of minimizing a convex function $f(x)$ subject to linear equality constraints:
\begin{equation}\label{eq:1}
\begin{array}{ll}
\mbox{minimize} & \displaystyle f(x):=g\left(x_1,\cdots, x_K\right)+\sum_{k=1}^{K}h_k(x_k)\\ [10pt]
\mbox{subject to} & \displaystyle E_1x_1+E_2x_2+\cdots+E_Kx_K=q,\\
& x_k\in X_k,\quad k=1,2,...,K,
\end{array}
\end{equation}
where $g(\cdot)$ is a smooth convex function; $h_k$ is a nonsmooth convex function; $x=(x_1^T,...,x_K^T)^T\in\Re^n$ is a partition of the optimization variable $x$, $x_k\in \Re^{n_k}$; $X=\prod_{k=1}^{K}X_k$ is the
feasible set for $x$; $q\in\Re^m$ is a vector. Let $E:=(E_1,\cdots, E_K)$ and $h(x):=\sum_{k=1}^{K}h_k(x_k)$. Many contemporary problems in signal processing, machine learning and smart grid systems can be formulated in the form \eqref{eq:1}. To motivate our work,  we discuss several examples of the form \eqref{eq:1}  below.

\subsection{Motivating Examples}\label{sec:example}
The first example is the basis pursuit (BP) problem which solves the following nonsmooth problem %\cite{ChenDonohoSaunders1998}
\begin{equation}\label{eq:BP}
\displaystyle{ {\operatornamewithlimits{{\min}}_{x}}} \  \|x\|_1\quad \st \quad E x = q, \ x\in X.
\end{equation}
One important application of this model is in compressive sensing, where a sparse signal (say an image) $x$ needs to be recovered using a small number of observations $q$ (i.e., $m\ll n$) \cite{ChenDonohoSaunders1998}. Let us partition the signal vector as $x=[x^T_1,\cdots, x^T_K]^T$ where $x_k\in \Re^{n_k}$, and partition $E$ and $X$ accordingly. Then the BP problem can be written in the form of \eqref{eq:1}
\begin{equation}\label{eq:BP_block}
\displaystyle{ {\operatornamewithlimits{\min}_{x}}}\  \sum_{k=1}^{K}\|x_k\|_1\quad \st \quad \sum_{k=1}^{K}E_k x_k = q,\; x_k\in X_k,\; \forall\; k.
\end{equation}

The second example has to do with the control of a smart grid system.
Consider a power grid system in which a utility company buys power from an electricity market to serve
a neighborhood with $K$ customers. The total cost for the utility includes the cost of purchasing the electricity from a
day-ahead wholesale market and a real-time market. In the envisioned smart grid system, the utility will have the ability to control the power consumption of some appliances (e.g., controlling the charging rate of electrical vehicles) in a
way to minimize its total cost. This problem, known as the demand response (DR) control problem,
is central to the success of the smart grid system \cite{Alizadeh12DM, Chang12, Li11demandresponse}.

To formulate this problem, let us divide each day into $L$ periods and let $p_\ell$ denote the amount of power
the utility company bids for the $\ell$-th period from a day-ahead market, $\ell=1,\cdots, L$.
Let $\bfPsi_k \bx_k$ denote the load profile of a customer $k=1,\cdots, K$,
where $\bx_k\in\mathbb{R}^{n_k}$ contains some control variables for the equipments of customer $k$,
and $\bfPsi_k\in\mathbb{R}^{L\times n_k}$ contains the information related to the appliance load model \cite{Paatero06}. The retailer aims at minimizing the bidding cost as well as the cost incurred by power imbalance in the next day \cite{Alizadeh12DM, Chang12, Li11demandresponse}
\begin{align}\label{eq:DR}
{\min}&\quad C_p
\bigg(\big(\sum_{k=k}^{K}\bfPsi_k \bx_k-\bp\big)^+\bigg)+C_s\bigg(\big(\bp-\sum_{k=1}^{K}\bfPsi_k \bx_k\big)^{+}\bigg)
+C_d(\bp)\nonumber\\
&\st\; \ \bx_k\in {X}_k, \ k=1,\cdots, K,\ \bx\ge 0, \ \boldsymbol{p}\ge 0
\end{align}
where $C_p(\cdot)$ and $C_s(\cdot)$ are increasing functions which model the cost incurred by
insufficient and excessive power bids, respectively; $C_d(\cdot)$ represents the bidding cost function;
$(x)^{+}:=\max\{x,0\}$; $X_k$ is some compact set; see \cite{Chang12}. Upon introducing a new variable
$\bz=\left(\sum_{k=1}^{K}\bfPsi_k \bx_k-\bp\right)^{+}$, the above problem can be equivalently transformed into the form of \eqref{eq:1}:
\begin{align}\label{eq:DREquiv}
\min&\quad C_p (\bz)+C_s\big(\bz+\bp-\sum_{k=1}^{K}\bfPsi_k \bx_k\big)+C_d(\bp)\\
\st& \quad \sum_{k=1}^{K}\bfPsi_k \bx_k-\bp-\bz\le 0,\; \bz\ge 0,\; \bp\ge 0, \; \bx_k\in X_k, \ \forall\ k. \nonumber
\end{align}

The third example is related to the optimization of the so-called cognitive radio network (CRN) \cite{zhao07, Scutari10MIMOCRN}. The CRN is an emerging wireless communication technology that promises significant improvement in radio spectrum utilization. The basic idea is to allow  secondary (unlicensed) users to opportunistically access the spectrum that is not used by primary (licensed) users. Specifically, consider a network with $K$ secondary users (SUs) and a single secondary base station (SBS) operating on $M$ parallel frequency tones. The SUs are interested in transmitting their messages to the SBS. Let $s^m_k$ denote user $k$'s transmit power on $m$th channel; let $h^m_k$ denote the channel between user $k$ and the SBS on tone $m$; let $P_k$ denote SU $k$'s total power budget. Also suppose that there are $L$ primary users (PUs) in the system, and let $g^{m}_{k\ell}$ denote the channel between the $k$th SU to the $\ell$th PU. The goal of the secondary network is to maximize the system throughput, subject to the requirement that certain interference temperature (IT) constraints measured at the receivers of the PUs are not violated \cite{fcc03a, hong11b}:
\begin{align}\label{problem:CRN}
\max&\quad \sum_{m=1}^{M}\log\left(1+\sum_{k=1}^{K}|h^m_k|^2 s^m_k\right)\\
\st&\quad s^m_k\ge 0, \ \sum_{m=1}^{M}s^m_k\le P_k,\ \sum_{k=1}^{K}|g^m_{k\ell}|^2s^m_k\le I^m_\ell,\ \forall~\ell,\ k,\ m\nonumber.
\end{align}
In the objective, the term $\log\left(1+\sum_{k=1}^{K}|h^m_k|^2 s^m_k\right)$ represents the sum-rate that all the users can jointly achieve on frequency tone $m$; $I^{m}_{\ell}\ge 0$ denotes the IT threshold for PU $\ell$ on tone $m$. Clearly this problem is also in the form of \eqref{eq:1}.

\subsection{Literature Review}
When the linear coupling constraint is not present, a
well known technique for solving \eqref{eq:1} is to use
the so-called block coordinate descent (BCD) method whereby, at every iteration, a single block
of variables is optimized while the remaining blocks are held fixed.
More specifically, at iteration~$r$, the blocks are updated in a Gauss-Seidel fashion by
\begin{equation}
\label{eq:GSperblock}
\begin{split}
x_k^r \in \arg\min_{x_k\in \mathcal{X}_k} \quad &g(x_1^{r},\ldots,x_{k-1}^{r},x_k,x_{k+1}^{r-1},\ldots, x_K^{r-1})+h_k(x_k), \ k=1,\cdots, K.\end{split}
\end{equation}

Since each step involves solving a simple subproblem of small size, the BCD method
can be quite effective for solving large-scale problems; see e.g., \cite{Friedman10, Razaviyayn12SUM, Saha10, shalev11, nestrov12} and the references therein.
The existing analysis of the BCD method \cite{tseng01,bertsekas96,bertsekas97, ortega72} requires the uniqueness of the minimizer for each subproblem \eqref{eq:GSperblock}, or the quasi convexity of $f$ \cite{Grippo00}. When problem \eqref{eq:GSperblock} is not easily solvable, a popular approach is to solve an approximate version of problem \eqref{eq:GSperblock}, yielding the block coordinate gradient decent (BCGD) algorithm (or the block coordinate proximal gradient algorithm in the presence of nonsmooth function $h$) \cite{tseng09coordiate, zhang13linear, shalev11, Beck13}. The global rate of convergence for BCD-type algorithm has been studied extensively. When the objective function is strongly convex, the BCD algorithm converges globally linearly \cite{luo93errorbound:10.1007/BF02096261}. When the objective function is smooth and not strongly convex, Luo and Tseng have shown that the BCD method and many of its variants can still converge linearly, provided that a certain local error bound condition is satisfied around the solution set \cite{Luo92CD,Luo92linear_convergence,Luo93dual, luo93errorbound:10.1007/BF02096261}. This line of analysis has recently been extended to allow a certain class of nonsmooth functions in the objective \cite{zhang13linear, tseng09approximation, Sanjabi13, hou13nips}. There are a few recent works characterizing the global sublinear convergence rate for the BCD-type algorithms \cite{Beck13, hong13complexity, nestrov12, richtarik12}. In particular, reference \cite{hong13complexity} shows that the BCD with Gauss-Seidel update rule converges sublinearly at the order of $\mathcal{O}(\frac{1}{r})$ for a large family of nonsmooth convex problems. Furthermore, a unified algorithmic framework called BSUM (block successive upper-bound minimization) and its convergence analysis is proposed in \cite{Razaviyayn12SUM} whereby at each step a locally tight upper-bound of the objective function is minimized successively to update the variable blocks.

%For more general problems where the objective is not strongly convex and the error bound condition does not hold, many recent works characterize the global sublinear convergence for various BCD-type algorithms. In \cite{nestrov12}, Nestrov shows that a randomized version of the BCGD algorithm, in which the coordinate is selected according to certain probability distribution, converges sublinearly in the order of $\mathcal{O}(1/r)$, where $r$ is the iteration number. In \cite{richtarik12, shalev11, lu13complexity} the authors show that the $\mathcal{O}(1/r)$ rate holds true for the more general settings with nonsmooth objective as well. When the coordinates are updated according to the traditional cyclic fashion, the literature on the iteration complexity for the BCD-type algorithm is more scant. In \cite{Saha10}, Saha and Tewari prove the $\mathcal{O}(1/r)$ rate for cyclic BCD algorithm when applied to certain special $\ell_1$ minimization problem. In \cite{Beck13}, Beck and Tetruashvili show the sublinear convergence for BCGD algorithm when applied to constrained smooth problem. Although experimentally, many authors have found that the cyclic BCD-type algorithm performs better than its randomized counterpart (see, e.g., \cite{Saha10}), to the best of our knowledge, the global rate of convergence for cyclic BCD-type algorithm is still open in the literature.

When the linear coupling constraint is present, it is well known that the BCD-type algorithm may fail to find any (local) optimal solution \cite{Solodov:1998:CCP:588896.589160}. A popular algorithm for solving this type of problem is the so-called alternating direction method of multipliers (ADMM) \cite{BoydADMM, EcksteinBertsekas1992, Glowinski1975, Glowinski1989}. In the ADMM method, instead of maintaining feasibility all the time, the constraint $Ex=q$ is dualized using the Lagrange multiplier $y$ and a quadratic penalty term is added. The resulting {\it augmented Lagrangian function} is of the form:
\begin{align}
L(x;y)= f(x) + \langle y,q-Ex \rangle + \frac{\rho}{2} \|q-Ex\|^2,\label{eq:aug-lagrangian}
\end{align}
where $\rho >0$ is a constant and $\langle \cdot,\cdot \rangle$ denotes the inner product operator.
The ADMM method updates the primal block variables $x_1,\ldots,x_n$ by using a BCD type procedure to minimize $L(x;y)$. The latter often leads to simple subproblems with closed form solutions. These primal updates are then followed by a gradient ascent update of the dual variable $y$.

Although the ADMM algorithm was introduced as early as in 1976
by Gabay, Mercier, Glowinski and Marrocco
\cite{Glowinski1975,ADMMGabbayMercier}, it has become popular
only recently due to its applications in modern large scale
optimization problems arising from machine learning and computer vision \cite{BoydADMM, Yin:2008:BIA:1658318.1658320, Yang09TV, zhang11primaldual,Scheinberg10inverse, TaoYuan2011}. In practice, the algorithm is often  computationally very efficient and exhibits much faster convergence than traditional algorithms such as the dual ascent algorithm \cite{bertsekas99, boyd04, Nedic09} or the method of multipliers \cite{bertsekas82}.
The convergence of ADMM has been established under the condition that the objective is separable and there are only two block variables, i.e., $g(x_1,\cdots, x_K)=g_1(x_1)+\cdots+g_K(x_K)$, and $K=2$ \cite{ADMMGabbayMercier, Glowinski1975}. For large scale problems such as those arising from compressive sensing, the optimal solution for the primal per-block subproblems may not be easily computable \cite{Yang_alternatingdirection}.  In these cases the classical ADMM can be modified to the one that performs a simple proximal gradient step for each subproblem \cite{Eckstein1994,EcksteinBertsekas1992,HeLiaoHanYang2002,WangYuan2012,zhang11primaldual, Yang_alternatingdirection}. When there are only two block variables, several recent works
\cite{HeYuan2012, Monteiro13} have shown that the ADMM method converges at a rate of
$\mathcal{O}(\frac{1}{r})$  (and $\mathcal{O}(\frac{1}{r^2})$ for
the accelerated version \cite{goldstein12}). Moreover, references
\cite{goldstein12,ADMMlinearBoley,ADMMlinearYin} have shown that
the ADMM converges linearly when the objective function is
strongly convex and there are only two blocks of variables. A recent study \cite{HongLuo2012ADMM} has shown the global (linear) convergence of the ADMM for the case of $K\ge 3$ under the assumptions that: {\it a)} for each $k$, $E_k$ is full column rank; {\it b)} the dual stepsize is sufficiently small; {\it c)} a certain error bound holds around the optimal solution set; and {\it d)} the objective is separable. If these conditions are not satisfied and when $K\ge3$, it is shown in \cite{chen13} that the ADMM can indeed diverge in general. Some other recent works have attempted to modify the original ADMM for $K\ge 3$ case \cite{he:alternating12, Wang13, ma12proximal}.

Unfortunately, neither BCD nor ADMM can be used to solve problem \eqref{eq:1}. In fact, due to its multi-block structure as well as the variable coupling in {\it both} the objective and the constraints, this problem cannot be handled by many other methods for big data including SpaRSA \cite{Wright09}, FPC-BB \cite{Hale08FixedPoint}, FISTA \cite{Beck:2009:FIS:1658360.1658364}, ALM \cite{goldfarb12}, HOGWILD \cite{niu11}, FPA \cite{scutari13flexible}. The main contribution of this paper is to propose and analyze a novel block successive upper bound minimization method of multipliers (BSUM-M) and its randomized version, that can solve problem \eqref{eq:1} efficiently. The BSUM-M algorithm integrates the BSUM and ADMM algorithm in a way that optimizes an \emph{approximate} augmented Lagrangian of the original problem one block variable each time, and then updates the dual variable by using a gradient ascent step. The resulting algorithm is flexible because we can choose suitable approximations of the augmented Lagrangian function that allow convenient updates of the primal variable blocks (say in closed form). In the absence of linear coupling constraints, the randomized BSUM-M algorithm reduces to the randomized BCD algorithm. In this case, we show that the randomized BCD algorithm in fact converges {linearly} (in expectation) for a family of problems without strongly convex objectives. To the best of our knowledge, this is the first result that shows the linear rate of convergence for the randomized BCD algorithm in the absence of strong convexity.

\subsection{The BSUM-M Algorithm}
Define $E:=(E_1,E_2,...,E_K)\in\Re^{m\times n}$, and $h(x):=\sum_{k=1}^{K}h_k(x_k)$.  %The augmented Lagrangian function of problem \eqref{eq:1} is of the form
%\begin{equation}\label{eq:aug-lagrangian}
%L(x;y) =  f(x)+ \langle y, q - Ex \rangle +\frac{\rho}{2}\|q-Ex\|^2,
%\end{equation}
%where  $\rho\ge0$ is a constant and $y$ is the multiplier associated with constraint $q=Ex$.
The augmented dual function is given by
\begin{equation}\label{1.1.1}
d(y) = \min_{x} \ g(x)+ \langle y, q - Ex \rangle +\frac{\rho}{2}\|q-Ex\|^2
\end{equation}
and the dual problem (equivalent to \eqref{eq:1} under mild conditions) is
\begin{equation}\label{2.3}
\max_y  \ d(y).
\end{equation}

In the following, we detail the proposed BSUM-M algorithm. In its simplest form, the BSUM-M algorithm  updates the dual variable using a gradient ascent step, followed by a BCD step for a certain approximate  version of the augmented Lagrangian \eqref{eq:aug-lagrangian}. In particular, at iteration $r+1$, the block variable $x_k$ is updated by solving the following subproblem
\begin{align}
\min_{x_k\in X_k}&\quad u_k\left(x_k; x^{r+1}_{1},\cdots, x^{r+1}_{k-1}, x^{r}_{k},\cdots, x^r_{K}\right)+\langle y^{r+1}, q-E_k x_k\rangle+h_k(x_k)\label{eq:BSUM}
\end{align}
where the function $u_k(\cdot\; ;\; x^{r+1}_{1},\cdots, x^{r+1}_{k-1}, x^{r}_{k},\cdots, x^r_{K})$ is an \emph{upper-bound} of $g(x)+\frac{\rho}{2}\|q-Ex\|^2$ at a given iterate $(x^{r+1}_{1},\cdots, x^{r+1}_{k-1}, x^{r}_{k},\cdots, x^r_{K})$.
To simplify notations, let us define a new set of auxiliary variables
\begin{align}
w^r_k&:=(x^{r}_1,\cdots, x^{r}_{k-1}, x^{r-1}_{k}, x^{r-1}_{k+1}, \cdots, x^{r-1}_K), \ k=1,\cdots, K,\nonumber\\
\ w^{r}_{K+1}&:=x^r,\quad  w^{r}_{1}:=x^{r-1} \nonumber.
%z^r_{k}&=[x^{r}_1,\cdots, x^{r}_{k-1}, x^{r}_{k}, x^{r-1}_{k+1}, \cdots, x^{r-1}_K]\in X, \ k=1,\cdots, K,\nonumber\\
%z^r_0&=x^{r-1}\in X, \ w^{r}_{K+1}=x^r\in X \nonumber.
\end{align}
%Clearly we have
%\begin{align}
%z^r_k-w^r_k&=[0,\cdots,0, x^{r}_{k}-x^{r-1}_{k},0, \cdots, 0], \\
%z^{r}_{k-1}&=w_k^{r},\ \forall~k=1,\cdots, K+1.\label{eq:z_w}
%\end{align}

The basic form of the BSUM-M algorithm is described in the following table.
\begin{center}
\fbox{
\begin{minipage}{6.3in}
\smallskip
\centerline{\bf Block Successive Upper-bound Minimization Method of Multipliers (BSUM-M)}
\smallskip
At each iteration $r\ge 1$:{
\begin{equation}\label{eq:BSUM-M}
\left\{\begin{array}{l}\displaystyle
\displaystyle y^{r+1}=y^r+\alpha^r(q-Ex^{r})=y^r+\alpha^r\left(q-\sum_{k=1}^KE_kx_k^{r}\right),\\[10pt]
x_k^{r+1}={\rm arg}\!\min_{x_k\in X_k}u_k(x_k; w_k^{r+1})-\langle y^{r+1}, E_k x_k \rangle+h_k(x_k),~\forall\ k
\end{array}
\right.
\end{equation}}
where $\alpha^r>0$ is the step size for the dual update.
\end{minipage}
}
\end{center}
In this paper, we also consider a randomized version of the BSUM-M algorithm whereby at each iteration either a single randomly chosen primal variable block or the dual variable is updated.
\begin{center}
\fbox{
\begin{minipage}{5.2in}
\smallskip
\centerline{\bf Randomized BSUM-M (RBSUM-M)}
\smallskip
Select a probability vector $\{p_k\}_{k=0}^{K}$ such that $p_k>0$ and $\sum_{k=0}^{K}p_k=1$.\\
At each iteration $t\ge 1$, pick an index $k\in \{0,\cdots, K\}$, with probability $p_k$, and{
\begin{equation}\label{eq:RBSUM-M}
\begin{array}{l}\displaystyle
{\bf If}\ k=0\\
\quad {y}^{t+1}=y^{t}+\alpha^{t}(q-Ex^{t}),\\[10pt]
\quad x_k^{t+1}=x_k^{t},\ k=1,\cdots, K.\\[10pt]
{\bf Else\ If}\  k\in\{1,\cdots, K\}\\[10pt]
\quad x_k^{t+1}={\rm arg}\!\min_{x_k\in X_k}u_k(x_k; x^{t})-\langle y^{t}, E_k x_k \rangle+h_k(x_k),\\[10pt]
\quad x_j^{t+1}=x_j^t,\ \forall~j\ne k, \\[10pt]
\quad y^{t+1}=y^t.\\
{\bf End}
\end{array}
\end{equation}}
where $\alpha^t>0$ is the step size for the dual update.
\end{minipage}
}
\end{center}
As explained in \cite{nestrov12, richtarik12}, the randomized version of the BCD-type algorithm is useful under many practical scenarios, for example when not all data is available at all times. We refer the readers to the aforementioned references for detailed discussions.
Note that here we have used the index $``t"$ to differentiate the iteration of RBSUM-M  with that of the BSUM-M. The reason is that in RBSUM-M, at each iteration only a single block variable (primal or dual) is updated, while in BSUM-M all primal and dual variables are updated once.

%Note that we have used the index $``t"$ to differentiate the iteration of RBSUM-M  with that of the BSUM-M. The reason is that in RBSUM-M, in each iteration only a single block variable (primal or dual) is updated, while for the BSUM-M all primal and dual variables are updated in one iteration.

\newsection{ Convergence Analysis} \label{sec:convergence}

\subsection{Main Assumptions} \label{sub:assumptions}
Suppose $f$ is a closed proper convex function in $\Re^n$.
Let ${\rm dom}\ f$ denote the effective domain of $f$ and let
$\hbox{int}(\hbox{dom } f)$ denote the interior of ${\rm dom}\ f$.
Let $x_{-k}$ (and similarly $E_{-k}$) denote the vector $x$ with $x_k$ removed.
We make the following standing assumptions regarding problem \eqref{eq:1}:
\pn {\bf Assumption A.}
\begin{itemize}
\item [(a)] Problem \eqref{eq:1} is a convex problem, its global minimum is attained and so is its dual optimal value. The intersection $X\cap \hbox{int}(\hbox{dom } f)\cap \{x\mid
Ex=q\}$ is nonempty.
%\item
%[(b)] The gradient of $g(\cdot)$ is block-coordinate-wise uniformly Lipschitz continuous
%\begin{align}
%\|\nabla_k g([x_{-k}, x_k])-\nabla_k g([x_{-k},\hat{x}_k])\|\le M_k \|x_k-\hat{x}_k\|,\quad\forall~x \in X, \ \forall~x_k, \hat{x}_k\in X_k,\ \forall\ k \label{eq:gk_lipchitz}
%\end{align}
%where $M_k>0$ is some constant. Define $M_{\rm max}:=\max_{k}M_k$.
%
%Then the gradient of $g(\cdot) $ is also uniformly Lipschitz continuous
%\begin{align}
%&\|\nabla g(x)-\nabla g(x')\|\le
%M\|x-x'\|,~\quad\quad\forall~x,x'\in X\label{eq:g_lipchitz}
%\end{align}
%where $M>0$ is a constant. %The functions $\ell_k(\cdot)$ is strictly convex and continuously differentiable
%%on $\hbox{int}(\hbox{dom } \ell_k)$ with a uniform Lipschitz continuous
%%gradient
%%\begin{align}
%%&\|B_k^T \nabla \ell_k(B_k x)-B_k^T\nabla \ell_k(B_k x')\|\le
%%L_\ell\|x_k-x_k'\|,~\quad\quad\forall~x_k,x_k'\in X\nonumber
%%\end{align}
%%where $L_k>0$ is a constant.

\item [(b)] The function $g(x)$ can be decomposed as $g(x)=\ell(Ax)+\langle x, b\rangle$, where $\ell(\cdot)$ is a strictly convex and continuously differentiable function on int(dom $\ell$), and $A$ is some given matrix (not necessarily full column rank).
%The gradient of $g(\cdot) $ is uniformly Lipschitz continuous
%\begin{align}
%&\|\nabla g(x)-\nabla g(x')\|\le
%L\|x-x'\|,~\quad\quad\forall~x,x'\in X\label{eq:g_lipchitz}
%\end{align}
%for some constant $L>0$.

Each nonsmooth function $h_k$, if present,  takes the  form
    $$h_k(x_k)=\lambda_k\|x_k\|_1+\sum_{J}w_J\|x_{k,J}\|_2,$$ where
$x_k=(\cdots, x_{k,J},\cdots)$ is a partition of $x_k$ with $J$
being the partition index; $\lambda_k\ge 0$ and $w_J\ge 0$ are some constants.

\item [(c)] The feasible sets $X_k$,
$k=1,\cdots,K$ are compact polyhedral sets, and are given by $X_k:=\{x_k \mid C_k x_k\le  c_k\}$, for some matrix $C_k\in\Re^{m_k\times n_k}$ and $c_k\in\Re^{m_k}$.
\end{itemize}

Next we make the following assumptions regarding the approximation function $u_k(\cdot;\cdot)$ in \eqref{eq:BSUM}.
\pn {\bf Assumption B.}
\begin{itemize}
\item [(a)]  $u_k(x_k; x)= g(x)+\frac{\rho}{2}\|Ex-q\|^2, \quad \forall\; x\in {X}, \ \forall\; k,$
\item [(b)] $u_k(v_k; x) \geq g(v_k,x_{-k})+\frac{\rho}{2}\|E_kv_k-q+E_{-k}x_{-k}\|^2,\quad\; \forall\; v_k \in {X}_k, \ \forall\; x \in{X}, \ \forall\; k,$
\item [(c)] $\nabla u_k(x_k;x)= \nabla_{k}\left(g(x)+\frac{\rho}{2}\|Ex-q\|^2\right), \quad \forall\; k, \ \forall\; x\in X,$
\item [(d)] For any given $x$, $u_k(v_k; x)$ is continuous in $v_k$ and $x$. It is also strongly convex in $v_k$, that is
$$u_k(v_k; x)\ge u_k(\hat{v}_k; x)+\langle\nabla u_k(\hat{v}_k; x),v_k-\hat{v}_k\rangle
+\frac{\gamma_k}{2}\|v_k-\hat{v}_k\|^2,\ \forall~v_k, \ \hat{v}_k\in X_k, \ \forall~x\in X$$
where $\gamma_k$ is independent of the choice of $x$.
\item [(e)] For given $x$, $u_k(v_k; x)$ has Lipchitz continuous gradient, that is
\begin{align}\label{eq:uk_lipchitz}
\|\nabla u_k(v_k; x)-\nabla u_k(\hat{v}_k; x)\|\le L_k\|v_k-\hat{v}_k\|,\ \forall\ \hat{v}_k,\ v_k\in X_k, \ \forall
\ k, \ \forall~x\in X,
\end{align}
where $L_k>0$ is some constant. Define $L_{\rm max}:=\max_{k}L_k$.
\end{itemize}

Below we give a few remarks about the assumptions made above.
\begin{remark}
The form of $g(\cdot)$ assumed in Assumption A(b) is fairly general. For example it includes the cases like $g(\cdot)=\sum_{k=1}^{K}\ell_k(A_k x_k)$, or $g(\cdot)=\ell(\sum_{k=1}^{K}A_k x_k)$, or the combination of these two, where $\ell_k(\cdot)$'s are strictly convex functions and $A_k$'s are matrices not necessarily with full rank. Moreover, since the matrix $A$ is not required to have full rank, $g(x)$ (hence $f(x)$) is not necessarily strongly convex with respect to $x$. Note that all three examples mentioned in Section \ref{sec:example} satisfy Assumption A(b). Moreover, this assumption requires that the nonsmooth function $h_k(\cdot)$ is in the form of mixed $\ell_1$ and $\ell_2$ norm.
\end{remark}

\begin{remark}
Assumption B indicates that for any $x$, each $u_k(\cdot; x)$ is an {\it upper-bound}, locally tight up to the first order,
for $g(x)+\frac{\rho}{2}\|q-Ex\|^2$ (the latter function itself satisfies Assumption B trivially). In many practical applications especially for nonsmooth problems, optimizing such functions often leads to much simpler subproblems than working directly
with the original function; see e.g., \cite{Razaviyayn12SUM,goldfarb12, WangYuan2012, zhang11primaldual}. As an example, suppose the augmented Lagrangian is given by:
    $$L(x;y)=\sum_{k=1}^{K}\|x_k\|_2+\langle y, q-Ax\rangle+\rho\|Ax-q\|^2.$$
    Then at $(r+1)$-th iteration, the subproblem for $x_k$ is given by
    $$x^{r+1}_k=\arg\min_{x_k\in X_k}\|x_k\|_2+\langle y^{r+1}, q-A_kx_k\rangle+\rho\|A_kx_k-d^{r+1}_k\|^2,$$
    for some constant $d^{r+1}_k=q-\sum_{j<k}A_j x^{r+1}_j-\sum_{j>k}A_jx_j^{r}$. This problem does not have closed form solution.
    A well-known strategy is to perform a proximal gradient step \cite{Combettes09}, that is, to
    solve the following approximate  problem instead
\begin{align}
\min_{x_k\in X_k}&\quad \|x_k\|_2+\langle y^{r+1}, q-A_kx_k\rangle+\langle2\rho A^T_k (A_kx^r_k-d^{r+1}_k), x_k \rangle +\frac{\tau}{2}\|x_k-x^r_k\|^2
\end{align}
This problem readily admits a closed form solution; see e.g. \cite{yuan06, zhang13linear}. Moreover, when choosing $\tau\ge \|A^T_k A_k\|$, the strongly convex function $\langle2\rho A^T_k (A_kx^r_k-d^{r+1}_k), x_k \rangle+\frac{\tau}{2}\|x_k-x^r_k\|^2$ is an approximation function that satisfies Assumption B (up to some constant).
\end{remark}

\begin{remark}
The strong convexity assumption for the approximation function $u_k(\cdot;\cdot)$ in B(d) is quite mild,
see the example given in the previous remark. This assumption ensures the iterates of (randomized) BSUM-M are well defined.
\end{remark}

\subsection{Preliminaries}\label{sub:preliminary}
We first provide two important results that characterize the augmented dual function $d(y)$ and the augmented Lagrangian function $L(x;y)$. Let $X(y)$ denote the set of optimal solutions for the primal problem, that is $$X(y):=\arg\min_{x} L(x; y).$$ Let $f^*$ denote the optimal value for \eqref{eq:1}. For any given $x\in X$ and any set $Y\subseteq X$, let ${\rm dist}(x, Y)$ denote the distance between $x$ and the set $Y$, that is, ${{\rm dist}(x, Y):=\min_{\hat{x}\in Y}\|x-\hat{x}\|}$.

The following lemma shows the differentiability and Lipchitz continuity of $d(y)$.
\begin{lemma}{\rm{(\cite[Lemma 2.1, 2.2]{HongLuo2012ADMM})}}\label{lm:const-derivative}
Suppose Assumption A holds. Then for any $y\in \Re^m$, both $Ex$ and $A_kx_k$, $k=1,2,...,K$, are constant over $X(y)$. Moreover, the dual function $d(y)$ is differentiable everywhere and
\[
\nabla d(y)=q-Ex(y),
\]
where $x(y)\in X(y)$.
Moreover, fix any scalar $\eta \le f^*$ and let ${\cal U} = \{\ y\in \Re^m\ |\
d(y) \ge \eta \ \}$.
Then there holds
\[
 \|\nabla d(y') - \nabla d(y)\|
\le \frac{1}{\rho}\|y' - y\|, \quad \forall\; y' \in {\cal U}, \ y
\in {\cal U}.
\]
\end{lemma}

We then introduce the notion of a proximal gradient, which will serve as a measure of optimality.
\begin{definition}{\rm{(Proximal\ Gradient)}}
Suppose a convex function $f(x)$ can be written as $f(x)=g(x)+h(x)$ where $g$ is convex and differentiable, $h$ is a convex (possibly nonsmooth) function. Then we can define the \textit{proximal gradient} of $f$ with respect to $h$ as
\[
\tilde \nabla f(\vx):=\vx-{\rm \prox}_{h}(\vx-\nabla g(\vx)),
\]
where $\prox_{h}(\cdot)$ is the proximity operator defined by
\[
{\rm \prox}_{h}(\vx)={\argmin_{u\in\Re^n}}\;\;h(u)+\frac12\|\vx-u\|^2.
\]
\end{definition}

Using the above definition, the proximal gradient for the augmented Lagrangian function can be expressed as
\begin{eqnarray}
\tilde \nabla_xL(x;y)&:=&x-\prox_h\left(x-\nabla_x(L(x;y)-h(x))\right)\label{eq:prox-grad}.
%&=&x-\prox_h\left(x-A^T\nabla g(Ax)+E^Ty-\rho E^T(Ex-q)\right)\label{eq:prox-grad}.
\end{eqnarray}

\begin{lemma} {\rm(\!\!\cite[Lemma 2.3]{HongLuo2012ADMM})}\label{lm:eb}
Suppose Assumptions A(a)---A(b) hold. Then
\begin{enumerate}
\item If in addition $X$ is a polyhedral set (not necessarily compact), then
there exist a positive scalars $\tau$ and $\delta$ such that the following error bound holds
\begin{equation}\label{eq:primaleb}
\dist(x,X(y))\le \tau \|\tilde\nabla_x L(x;y)\|,\quad
\end{equation}
for all $(x,y)$ such that $\|\tilde\nabla_x L(x;y)\|\le \delta$, where the proximal gradient $\tilde \nabla_x L(x;y)$ is given by
\eqref{eq:prox-grad}.

\item If  $X$ is also a compact set, then there exists some $\tau>0$ such that
the error bound \eqref{eq:primaleb} holds for all $x\in X\cap {\rm dom} ~h$.
%\item
%Similarly, if assumption A(e) also holds, then for any scalar $\zeta$, there exist
%positive scalars $\delta$ and $\tau$ such that
%\begin{equation}\label{eq:dualeb}
%\dist(y,Y^*)=\|y-y^*\|\le \tau\|\nabla d(y)\|,~\mbox{whenever}~
%d(y)\ge \zeta~\mbox{and}~\|\nabla d(y)\|\le \delta.
%\end{equation}
\end{enumerate}
Moreover, in both cases the constant $\tau$ is \emph{independent} of
the choice of $y$ and $x$.
\end{lemma}
If either the objective function $f$ is strongly convex (i.e., $A$ is full column rank in Assumptions A(b)), or if $E$ is full row rank, then the augmented Lagrangian function $L(x;y)$ is strongly convex. In this case, the error bound in Lemma~\ref{lm:eb} holds automatically and globally with Assumption A(b) or the requirement that $X_k$ being polyhedral or compact.

\subsection{Convergence Analysis}\label{sub:convergence}
In this section, we analyze the convergence of BSUM-M and its randomized version RBSUM-M.

Under Assumption B(d), each function $u_k(v_k; x)$
is strongly convex with respect to $v_k\in X_k$. As a result, the primal update steps in
\eqref{eq:BSUM-M} and \eqref{eq:RBSUM-M} are both well defined and
%\[x_k^{r+1}={\rm arg}\!\min_{x_k\in X_k}u_k(x_k; y^r, x_1^{r+1},...,x^{r+1}_{k-1},x^r_{k},...,x^r_K)+h_k(x_k)\]
have unique optimal solutions.
For RBSUM-M, let us define a new vector $\hat{x}^{t+1}=[\hx^{t+1}_1, \cdots, \hx^{t+1}_K]$ where
\begin{align}
\hx_k^{t+1}:={\rm arg}\!\min_{x_k\in X_k}u_k(x_k; x^t)+\langle y^t, q-E_k x_k\rangle+h_k(x_k), \quad k=1,...,K.\label{eq:defxhat}
\end{align}
Define $\hat{y}^{t+1}$ as
\begin{align}\label{eq:defineyhat}
\hy^{t+1}=y^t+\alpha^{t}\left(q-Ex^{t}\right).
\end{align}
Define $z^t:=[x^t_1, \cdots, x_K^t, y^{t}]$, then we can write $L(x^t; y^{t})$ equivalently as $L(z^t)$.

We first characterize the successive difference of the augmented Lagrangian before and after one primal update (resp. one update) for the BSUM-M (resp. RBSUM-M).
\begin{lemma}\label{lm:p-descent}
Suppose Assumption B holds. Then
\begin{enumerate}
\item For BSUM-M, we have
\begin{equation}\label{eq:p-descent-bcdmm}
L(\vx^r;y^{r+1})- L(\vx^{r+1};y^{r+1}) \ge \gamma\|x^r - x^{r+1}\|^2,
\end{equation}
where the constant $\gamma>0$ is independent of $r$ and $y^{r+1}$.
\item For RBSUM-M, we have
\begin{equation}\label{eq:p-descent-sbcdmm}
\mathbb{E}[L(z^t)- L(z^{t+1})\mid z^t] \ge \hat{\gamma}\|\vx^t - \hx^{t+1}\|^2-\alpha^t p_0\|q-Ex^t\|^2,
\end{equation}
where the expectation is taken over the algorithm's random choice of the update index; the constant $\hat{\gamma}>0$ is independent of $t$ and $y^t$.
\end{enumerate}
\end{lemma}
\proof We first show part (1) of the claim. Using Assumption B, we have that
\begin{align}
&L(w^{r+1}_k;y^{r+1})- L(w^{r+1}_{k+1};y^{r+1})\nonumber\\
&\ge u_k(\vx^{r}_{k} ; w^{r+1}_k)-\langle y^{r+1}, E_k x^r_k \rangle+h_k(x^r_k)\nonumber\\
&\quad-\left(u_k(\vx^{r+1}_{k} ; w^{r+1}_k)-\langle y^{r+1}, E_k x^{r+1}_k \rangle+h_k(x^{r+1}_k)\right)\nonumber\\
&\ge \gamma_k\|\vx^{r+1}_k-\vx^r_k\|^2
\end{align}
where the first inequality is due to Assumption B(a)--B(b), the second inequality is due to the strong convexity Assumption B(d), and the fact that $x_k^{r+1}$ is the optimal solution for the convex problem
$${\rm arg}\!\min_{x_k\in X_k}u_k(x_k; w^{r+1}_k)-\langle y^{r+1}, E_k x_k\rangle+h_k(x_k).$$
Summing over $k$ and letting $\gamma:=\min_{k}\gamma_k$, we obtain
\begin{align}
L(\vx^r;y^{r+1})- L(\vx^{r+1};y^{r+1})\ge \gamma\|x^r - x^{r+1}\|^2.
\end{align}

We then show part (2) of the claim. We have the following
\begin{align}\label{eq:older}
&\mathbb{E}[L(z^t)- L(z^{t+1})\mid z^t] \nonumber\\
&=\sum_{k=1}^{K} p_k\left[L(x^t,y^{t})-L(x^t_{-k},\hat{x}^{t+1}_k; y^{t})\right]+p_0\left[L(x^t; y^{t})-L(x^t; \hy^{t+1})\right]\nonumber\\
&\ge \sum_{k=1}^{K} p_k\left[u_k(x^t_k; x^t)-\langle y^t, E_k x^t_k\rangle+h_k(x^t_k)-u(\hx^{t+1}_k; x^t)+\langle y^t, E_k \hat{x}^{t+1}_k\rangle-h_k(\hat{x}^{t+1}_k)\right]\nonumber\\
&\quad\quad+p_0\left[L(x^t; y^{t})-L(x^t; \hy^{t+1})\right]\nonumber\\
&\ge \sum_{k=1}^{K}p_k\gamma_k\|x^t_k-\hat{x}^{t+1}_k\|^2-\alpha^tp_0\|q-Ex^{t}\|^2  \nonumber\\
&\ge \hat{\gamma}\|\vx^t - \hx^{t+1}\|^2-\alpha^t p_0\|q-Ex^{t}\|^2
\end{align}
where $\hat{\gamma}:=\min_k p_k \gamma_k$ is independent of $t$ and $y^t$.
\QED

Next we bound the size of the proximal gradient at any given iterate.
\begin{lemma} \label{lm:estimate}
{Suppose %Assumptions A(b) and
Assumption B holds. Then
\begin{enumerate}
\item For the iterates $\{(x^r, y^r)\}$ generated by the BSUM-M, there exists some constant
$\sigma>0$ $($independent of $y^r$$)$ such that
\begin{equation}\label{eq:size_prox_bcdmm}
\|\tilde\nabla L(x^r;y^r)\|\le \sigma\|x^{r+1}-x^r\|
\end{equation}
for all $r\ge1$.
\item For the iterates $\{(x^t, y^t)\}$ generated by the RBSUM-M, there exist some constants
$\hat{\sigma}_1>0$ and $\hat{\sigma_2}>0$ $($independent of $y^t$$)$ such that
\begin{equation}\label{eq:size_prox_sbcdmm}
\|\tilde\nabla L(x^t;\hy^{t+1})\|\le \hat{\sigma}_1\|\hx^{t+1}-x^t\|+\hat{\sigma}_2\|\hy^{t+1}-y^t\|
\end{equation}
for all $t\ge1$.
\end{enumerate}
}
\end{lemma}
\proof
The proof of two cases follow similar steps, thus we only prove the second case here. Fix
any $t\ge1$ and any $1\le k\le K$. According to the definition of $\hx_k^{t+1}$ in \eqref{eq:defxhat},
we have
\begin{equation}\label{eq:old}
\! \hx_k^{t+1}=\prox_{h_k}\left[\hx_k^{t+1}-\nabla u_k(\hx_k^{t+1}; x^{t})+E^T_ky^t \right].
\end{equation}
Therefore, we have
\begin{align}\label{eq:error_norm}
&\left\|\hx_k^{t+1}-\prox_{h_k}\left[x_k^{t}-\nabla u_k(x_k^{t}; x^{t})+E^T_k\hy^{t+1} \right]\right\|\nonumber\\
&=\left\|\prox_{h_k}\left[\hx_k^{t+1}-\nabla u_k(\hx_k^{t+1}; x^{t})+E^T_ky^t \right]-\prox_{h_k}\left[x_k^{t}-\nabla u_k(x_k^{t}; x^{t})+E^T_k\hy^{t+1} \right]\right\|\nonumber\\
%&\le \|\hx_k^{t+1}-x_k^{t}\|+\|\nabla u_k(\hx_k^{t+1}; y^{t}, x^{t})-\nabla u_k(x_k^{t}; \hy^{t+1}, x^{t})\|\nonumber\\
&\le (L_k+1)\|\hx_k^{t+1}-x_k^{t}\|+\|E_k\|\|\hy^{t+1}-y^t\|\nonumber
\end{align}
where the inequality follows from the nonexpansive property of the prox operator, as well as the Lipschitz continuity property of the gradient vector $\nabla u_k$ (cf.\ Assumption~B(e)).
Using this relation and the definition of the proximal gradient $\tilde\nabla_k L(x^t;\hy^{t+1})$, we have
\begin{eqnarray*}
\|\tilde\nabla_{k} L(x^t;\hy^{t+1})\|&=& \left\|x^t_k-\prox_{h_k}\left(x_k^{t}-\nabla_{k}g(x^t)+E^T_k \hy^{t+1}\right)\right\|\\
&\le &\|x^t_k-\hx^{t+1}_k\|+ \left\|\hx_k^{t+1}-
\prox_{h_k}\left(x_k^{t}-\nabla_{k}g(x^t)+E^T_k \hy^{t+1}\right)\right\|\\
&=&\|x^t_k-\hx^{t+1}_k\|+ \left\|\hx_k^{t+1}-
\prox_{h_k}\left(x_k^{t}-\nabla_{k}u_k(x^t_k; x^t)+E^T_k \hy^{t+1}\right)\right\|\\
&\le & (L_k+2)\|\hx^{t+1}-x^t\|+\|E_k\|\|\hy^{t+1}-y^t\|,\quad \forall \ k=1,2,...,K.
\end{eqnarray*}
This further implies that the full proximal gradient vector can be bounded by $\|\hx^{t+1}-x^t\|$:
\[
\|\tilde\nabla L(x^t;\hy^{t+1})\|\le \left(\max_k\{L_k\}+2\right)\sqrt{K}\|\hx^{t+1}-x^t\|+\sqrt{K}\max_{k}\|E_k\|\|\hy^{t+1}-y^t\|.
\]
Setting $\hat{\sigma}_1=(\max_{k}\{L_k\}+2)\sqrt{K}$ and $\hat{\sigma}_2=\sqrt{K}\max_{k}\|E_k\|$ (both of which are independent of $y^t$) completes the proof.
\QED

To analyze the convergence of the algorithms, we need to make use of a certain ``potential function" that measures the algorithm progress. Similar to \cite{HongLuo2012ADMM}, we will adopt the combined primal and dual optimality gap (to be defined shortly) as the ``potential function".

Let $d^*$ denote the dual optimal value. Due to
Assumption A(a), $d^*$ also equals to the primal optimal value.
For each algorithm, define the {\it dual optimality gap} by
\begin{equation}\label{eq:dd-gap}
\Delta_d^r=d^*-d(y^r), \quad \Delta_d^t=d^*-d(y^t),
\end{equation}
each of which represents the gap to the dual optimality at the current iteration. Similarly, define
the {\it primal optimality gap} at each iteration by
\begin{equation}\label{eq:pp-gap}
\Delta_p^r=L(x^{r};y^{r})-d(y^{r}),\quad
\Delta_p^t=L(x^{t};y^{t})-d(y^{t}).
\end{equation}
Clearly, we have both $\Delta_d^r\ge0$ and $\Delta_p^r\ge0$ for all
$r\ge 1$ (resp. $\Delta_d^t\ge0$ and $\Delta_p^t\ge0$  for all $t\ge 1$).

Let $X(y^r)$ denote the set of optimal solutions for the following
optimization problem
\[
\min_x L(x;y^r)=\min_x g(x)+\langle y^r, q - Ex \rangle
+\frac{\rho}{2}\|Ex-q\|^2.
\]
We denote
\[
\barx^r=\argmin_{\barx\in X(y^r)}\|\barx-x^r\|,\quad \barhx^{t}=\argmin_{\barhx\in X(\hat{y}^{t})}\|\barhx-x^t\|.
\]

We then bound the decrease of the dual optimality gap for BSUM-M as well as the conditional expected decrease of the dual optimality gap for the RBSUM-M.
\begin{lemma}\label{lm:dual-gap}
\begin{enumerate}
\item For the BSUM-M algorithm, there holds
\begin{equation}
\Delta_d^r-\Delta_d^{r-1}\le -\alpha^{r-1}(Ex^{r}-q)^T(E\barx^{r}-q). \label{eq:dual-gap-BSUM-M}
\end{equation}

\item For the RBSUM-M algorithm, there holds
\begin{equation}
\mathbb{E}[\Delta_d^t-\Delta_d^{t-1}\mid z^{t-1}]\le -\alpha^{t-1}p_0(Ex^{t-1}-q)^T(E\barhx^{t}-q). \label{eq:dual-gap-SBSUM-M}
\end{equation}
\end{enumerate}
\end{lemma}
\proof
The proof for the first case is similar to \cite[Lemma 3.2]{HongLuo2012ADMM}. We outline the proof here for completeness. We have the following series of inequalities
\begin{eqnarray}
\Delta_d^r-\Delta_d^{r-1}&=& d(y^{r-1})-d(y^{r})\nonumber\\
&=&L(\barx^{r-1};y^{r-1})-L(\barx^{r};y^{r})\nonumber\\
&=&[L(\barx^{r};y^{r-1})-L(\barx^{r};y^{r})] + [L(\barx^{r-1};y^{r-1})-L(\barx^{r};y^{r-1})]\nonumber \\
&=&(y^{r-1}-y^{r})^T(q-E\barx^{r})+[L(\barx^{r-1};y^{r-1})-L(\barx^{r};y^{r-1})]\nonumber\\
&= &-\alpha^{r-1}(Ex^{r-1}-q)^T(E\barx^{r}-q)+[L(\barx^{r-1};y^{r-1})-L(\barx^{r};y^{r-1})]\nonumber\\
&\le &-\alpha^{r-1}(Ex^{r-1}-q)^T(E\barx^{r}-q), \quad \forall\; r\ge 1, \label{eq:successive_dual_difference}%\label{eq:dual-gap}
\end{eqnarray}
where the last equality follows from the update of the dual variable
$y^{r-1}$; the last inequality is due to the fact that $\bar{x}^{r-1}$ minimizes $L(\cdot,
y^{r-1})$.

The proof for the second case is straightforward, as we can readily observe that
\begin{align}
\mathbb{E}\left[\Delta_d^t-\Delta_d^{t-1}\mid z^{t-1}\right]&=\mathbb{E}\left[d(y^{t-1})-d(y^{t})\mid z^{t-1}\right]\nonumber\\
&=p_0 \left(d(y^{t-1})-d(\hy^{t})\right)\nonumber\\
&\le -\alpha^{t-1} p_0 (Ex^{t-1}-q)^T(E\barhx^{t}-q),\nonumber
\end{align}
where the last inequality has utilized the result in \eqref{eq:successive_dual_difference}.
This concludes the proof. \QED

Next we proceed to bound the decrease (resp. conditional expected decrease) of the primal gap for the BSUM-M (resp. RBSUM-M).
\begin{lemma}\label{lm:primal-descent}
{Suppose Assumption B holds. Then

\begin{enumerate}
\item For the BSUM-M, the following bound holds true for each $r\ge 1$
\begin{equation}
\Delta_p^{r+1}-\Delta_p^{r}\le \alpha^r\|Ex^r-q\|^2-{\gamma}\|x^{r+1}-x^r\|^2-\alpha^r(Ex^{r}-q)^T(E\barx^{r+1}-q)\label{eq:primal-gap-bcdmm}
\end{equation}
for some ${\gamma}$ independent of $y^r$.

\item For the RBSUM-M, the following bound holds true for each $t\ge 1$
\begin{equation}
\mathbb{E}\left[\Delta_p^{t+1}-\Delta_p^{t}\mid z^t\right]\le p_0 \alpha^r\|Ex^t-q\|^2-\hat{\gamma}\|\hx^{t+1}-x^t\|^2-\alpha^t p_0 (Ex^{t}-q)^T(E\barhx^{t+1}-q)\label{eq:primal-gap-sbcdmm}
\end{equation}
for some $\hat{\gamma}$ independent of $y^r$.
\end{enumerate}

}
\end{lemma}
\proof We first show part (1) of the claim. This result is similar to \cite[Lemma 3.3]{HongLuo2012ADMM}, and we include its derivation here for completeness.  Fix any $r\ge1$,
by using the dual update rule (cf.\ \eqref{eq:BSUM-M}), we have
\begin{align}
L(x^{r};y^{r+1})&=f(x^{r})+\langle y^{r}, q-Ex^{r}\rangle +\frac{\rho}{2}\|Ex^{r}-q\|^2+\alpha^r\|Ex^r-q\|^2\nonumber\\
&=L(x^r;y^{r})+\alpha^r\|Ex^r-q\|^2\nonumber.
\end{align}
Combined with the first part of Lemma~\ref{lm:p-descent}, we obtain
\[
L(x^{r+1};y^{r+1})-L(x^r;y^{r})\le \alpha^r\|Ex^r-q\|^2-\gamma\|x^{r+1}-x^r\|^2, \quad \forall\; r\ge1.
\]
Hence, we have the following bound on the reduction of primal optimality gap
\begin{eqnarray*}
\Delta_p^{r+1}-\Delta_p^{r}&=&(L(x^{r+1};y^{r+1})-d(y^{r+1}))-(L(x^r;y^{r})-d(y^{r}))\nonumber\\
&=& L(x^{r+1};y^{r+1})-L(x^r;y^{r})-(d(y^{r+1})-d(y^{r}))\nonumber\\
&\le & \alpha^r\|Ex^r-q\|^2-\gamma\|x^{r+1}-x^r\|^2-\alpha^r(Ex^{r}-q)^T(E\barx^{r+1}-q), \quad \forall \; r\ge1,
\end{eqnarray*}
where the last step is due to the first part of Lemma~\ref{lm:dual-gap}.

We then show part (2) of the claim. We have that for all $t\ge1$
\begin{eqnarray*}
\mathbb{E}\left[\Delta_p^{t+1}-\Delta_p^{t}\mid z^t\right]
&=&\mathbb{E}\left[(L(z^{t+1})-d(y^{t+1}))-(L(z^t)-d(y^{t}))\mid z^t\right]\nonumber\\
&=& \mathbb{E}[L(z^{t+1})-L(z^t)|z^t]-\mathbb{E}[d(y^{t+1})-d(y^{t})\mid z^t]\nonumber\\
&\le& -\hat{\gamma}\|\hx^{t+1}-x^t\|^2+\alpha^t p_0\|Ex^t-q\|^2-\alpha^tp_0(Ex^{t}-q)^T(E\barhx^{t+1}-q)\nonumber
\end{eqnarray*}
where the last step is due to Lemma~\ref{lm:dual-gap} and Lemma~\ref{lm:p-descent}.\QED

Next we present the first main result regarding the convergence of the BSUM-M and RBSUM-M.
\begin{theorem}\label{thm:main}
Suppose that the error bound in Lemma~\ref{lm:eb} and Assumption B hold.
Assume that one of the following stepsize rules is used:  {\it i)} for all $r$, the stepsize $\alpha^r=\alpha$ is sufficiently small, or {\it ii)} $\alpha^r$ satisfies
\begin{align}
\sum_{r=1}^{\infty}\alpha^r=\infty, \quad \lim_{r\to\infty}\alpha^r=0.  \label{eq:step_size} %, \quad \alpha^{r}\le \alpha^{r-1}.  \label{eq:step_size}
\end{align}
Then we have the following:{
\begin{enumerate}
\item For the BSUM-M, we have $\lim_{r\to\infty}\|Ex^r-q\|=0$, $\lim_{r\to\infty}\|x^r-{x}^{r+1}\|=0$ and \quad $\lim_{r\to\infty}\|x^r-\bar{x}^r\|=0$. Further, every limit point of $\{x^r,y^r\}$ is a primal and dual optimal solution.
\item For the RBSUM-M, we have $\lim_{t\to\infty}\|Ex^t-q\|=0$, $\lim_{t\to\infty}\|x^t-{x}^{t+1}\|=0$, and \quad $\lim_{t\to\infty}\|x^t-\bar{x}^t\|=0$  w.p.1. Further, every limit point of $\{x^t,y^t\}$ is a primal and dual optimal solution w.p.1.
\end{enumerate}}
\end{theorem}
\proof
We focus on showing part (2). The proof for part (1) is easier and follows similar steps. %We show by induction that the expectation of the sum of optimality gaps $\Delta^r_d+\Delta^r_p$ is reduced after each ADMM iteration, as long as the stepsize $\alpha$ is chosen sufficiently small. For any $r\ge1$, we denote
From Assumption A(c) we have that each $X_k$ is compact, which implies the boundedness of $x^t$. Thus, we obtain from Lemma \ref{lm:eb} that
\begin{equation}\label{eq:bds}
\|x^t-\barhx^{t+1}\|\le \tau \|\tilde \nabla L(x^t;\hy^{t+1})\|
\end{equation}
for some $\tau>0$ (independent of $y^t$).
Combining the two estimates \eqref{eq:dual-gap-SBSUM-M} and \eqref{eq:primal-gap-sbcdmm}, we obtain
\begin{eqnarray}
&&\mathbb{E}[\Delta_p^{t+1}+\Delta_d^{t+1}\mid z^t]-\mathbb{E}[\Delta_p^{t}+\Delta_d^{t}\mid z^t]\\
&=&
\mathbb{E}[\Delta_p^{t+1}-\Delta_p^{t}\mid z^t]+\mathbb{E}[\Delta_d^{t+1}-\Delta_d^{t}\mid z^t]\nonumber\\
&\le & \alpha^t p_0\|Ex^t-q\|^2-\hat{\gamma}\|\hx^{t+1}-x^t\|^2-2\alpha^t p_0(Ex^{t}-q)^T(E\barhx^{t+1}-q)\nonumber\\
&=&\alpha^tp_0\|Ex^t-E\barhx^{t+1}\|^2-\alpha^tp_0\|E\barhx^{t+1}-q\|^2-{\hat\gamma}\|\hx^{t+1}-x^t\|^2.\label{eq:estimate}
\end{eqnarray}
Now we invoke \eqref{eq:bds} and Lemma~\ref{lm:estimate} to upper bound $\|x^t-\barhx^{t+1}\|$:
\begin{equation}\label{eq:nice1}
\|x^t-\barhx^{t+1}\|\le\tau\|\tilde \nabla L(x^t;\hy^{t+1})\|\le\tau \left(\hat{\sigma}_1 \|\hx^{t+1}-x^t\|+\hat{\sigma}_2\|\hy^{t+1}-y^t\|\right).
\end{equation}
Therefore, defining $\tilde{\sigma}^2_i= 2\tau^2\hat{\sigma}^2_i$ ($i=1,2$),  we have
\begin{align}\label{eq:nice}
\|x^t-\barhx^{t+1}\|^2&\le2\tau^2 \left(\hat{\sigma}^2_1 \|\hx^{t+1}-x^t\|^2+\hat{\sigma}^2_2\|\hy^{t+1}-y^t\|^2\right)\nonumber\\
&:=\tilde{\sigma}^2_1 \|\hx^{t+1}-x^t\|^2+\tilde{\sigma}^2_2\|\hy^{t+1}-y^t\|^2\nonumber\\
&=\tilde{\sigma}^2_1 \|\hx^{t+1}-x^t\|^2+\tilde{\sigma}^2_2\alpha_t^2\|q-Ex^t\|^2,
\end{align}
where the last step follows from \eqref{eq:defineyhat}.
Using this result, we can bound the size of the constraint violation as follows
\begin{align}
\|q-Ex^t\|^2&=\|q-E\barhx^{t+1}+E\barhx^{t+1}-Ex^t\|^2\nonumber\\
&\le2\|q-E\barhx^{t+1}\|^2+2\|E\|^2\|\barhx^{t+1}-x^t\|^2\nonumber\\
&\le2\|q-E\barhx^{t+1}\|^2+ 2\|E\|^2\left(\tilde{\sigma}^2_1 \|\hx^{t+1}-x^t\|^2+\tilde{\sigma}^2_2\|\hy^{t+1}-y^t\|^2\right)\nonumber\\
&=2\|q-E\barhx^{t+1}\|^2+ 2\|E\|^2\left(\tilde{\sigma}^2_1 \|\hx^{t+1}-x^t\|^2+\tilde{\sigma}^2_2(\alpha^t)^2\|q-Ex^t\|^2\right)\nonumber.
\end{align}
Rearranging terms, we obtain (assuming $\alpha^t$ is small enough such that $1-2\tilde{\sigma}^2_2(\alpha^t)^2\|E\|^2>0$)
\begin{align}
\|q-Ex^t\|^2&\le\frac{2\|q-E\barhx^{t+1}\|^2+ 2\|E\|^2\tilde{\sigma}^2_1 \|\hx^{t+1}-x^t\|^2}{1-2\tilde{\sigma}^2_2(\alpha^t)^2\|E\|^2}\label{eq:bound_qEx}.
\end{align}

Substituting \eqref{eq:nice} into \eqref{eq:estimate} and using \eqref{eq:bound_qEx} yields
\begin{align}\label{eq:descent}
&\mathbb{E}\left[(\Delta_p^{t+1}+\Delta_d^{t+1})-(\Delta_p^{t}+\Delta_d^{t})\mid z^t\right] \nonumber\\
&\le
(\alpha^tp_0\|E\|^2\tilde{\sigma}_1^{2}-\hat{\gamma})\|\hx^{t+1}-x^t\|^2-\alpha^tp_0\|E\barhx^{t+1}-q\|^2
+(\alpha^t)^3p_0\|E\|^2\tilde{\sigma}^2_2\|q-Ex^t\|^2\nonumber\\
&\le
\left(\alpha^tp_0\|E\|^2\tilde{\sigma}_1^{2}+\frac{2\|E\|^4 \tilde{\sigma}^2_1 (\alpha^t)^3p_0\tilde{\sigma}^2_2}{1-2\tilde{\sigma}^2_2(\alpha^t)^2\|E\|^2}-\hat{\gamma}\right)\|\hx^{t+1}-x^t\|^2\nonumber\\
&\quad+\left(\frac{2(\alpha^t)^3p_0\|E\|^2\tilde{\sigma}^2_2}{1-2\tilde{\sigma}^2_2(\alpha^t)^2\|E\|^2}-\alpha^tp_0\right)\|E\barhx^{t+1}-q\|^2.
\end{align}

{\bf Case 1)}: If we choose the constant stepsize $\alpha^t=\alpha$, and let $\alpha$ be sufficiently small.
Then the constants in front of $\|\hx^{t+1}-x^t\|^2$ and $\|E\barhx^{t+1}-q\|^2$ in \eqref{eq:descent} become negative. By applying the convergence theorem of non-negative almost supermartingale  \cite[Theorem 1]{Robbins71}, we have that
\begin{align}
&\lim_{t\to\infty}\Delta_p^{t+1}+\Delta_d^{t+1}  \quad \textrm {exists and is finite, w.p.1,}\nonumber\\
&\lim_{t\to\infty}\|\hx^{t+1}-x^t\|= 0 \quad\textrm {w.p.1,}\label{eq:primal_difference_zero}\\
&\lim_{t\to\infty}\|E\barhx^{t+1}-q\|=\lim_{t\to\infty}\|\nabla d(y^{t+1})\|=0 \quad\textrm {w.p.1.}\label{eq:dual_gap_zero}
\end{align}
{
We conclude that every limit point of the sequence $\{y^{t}\}$ is a dual optimal solution.
%Applying \eqref{eq:primal_difference_zero} and \eqref{eq:dual_gap_zero} to \eqref{eq:bound_qEx}, we have
%\begin{align}
%&\lim_{t\to\infty}\|E x^{t}-q\|= 0, \ \mbox{w.p.1}, \nonumber\\
%&\lim_{t\to\infty}\|y^{t}-\hy^{t+1}\|= 0, \ \mbox{w.p.1} \label{eq:y_difference}
%\end{align}
%which further implies that
%\begin{align}
%\lim_{t\to\infty}\|\nabla d(\hy^{t+1})\|=0 \quad\textrm {w.p.1.}
%\end{align}
Further, by \eqref{eq:bound_qEx}, the  constraint violation vanishes in the limit, i.e.,
\begin{align}
&\lim_{t\to\infty}\|E x^{t}-q\|=\lim_{t\to\infty}\|y^{t}-\hy^{t+1}\|= 0, \ \mbox{w.p.1.} \label{eq:y_difference}
\end{align}
Using \eqref{eq:primal_difference_zero} and the fact that $x^{t+1}-x^t$ has only one nonzero block which equals the corresponding block of
$\hx^{t+1}-x^t$ (c.f.\ \eqref{eq:defxhat}), we have
\begin{align}
&\lim_{t\to\infty}\|x^{t+1}-x^t\|\le \lim_{t\to\infty}\|\hx^{t+1}-x^t\|= 0,  \ \mbox{w.p.1} \label{eq:primal_difference_zero2}.
\end{align}
Substituting \eqref{eq:primal_difference_zero} and \eqref{eq:y_difference} into \eqref{eq:nice1}, we obtain
\(\displaystyle
\lim_{t\to\infty}\|x^t-\bar{x}^{t+1}\|=0,\ \mbox{w.p.1.}
\)
Combining this with \eqref{eq:primal_difference_zero2}  further implies
\begin{align}
&\lim_{t\to\infty}\|x^{t+1}-\bar{x}^{t+1}\|= 0,\ \mbox{w.p.1.}\label{eq:primal_gap_zero}
\end{align}
Since $\bar{x}^{t+1}\in X (y^{t+1})$, we have $L(\bar{x}^{t+1}, y^{t+1})\le L(x, y^{t+1})$ for all $x\in X$. Passing limit, we have $L({x}^{\infty}, y^{\infty})\le L(x, y^{\infty})$ for all $x\in X$ w.p.1, where $({x}^{\infty}, y^{\infty})$ is a limit point of $\{{x}^t, y^t\}$. Combining this with \eqref{eq:y_difference}, we conclude that ${x}^{\infty}$ is a primal optimal solution satisfying ${x}^{\infty}\in X(y^{\infty})$ w.p.1.}

%Next we show that \eqref{eq:primal_gap_zero} and \eqref{eq:y_difference} guarantee
%\begin{align}
%\lim_{t\to\infty}\left(\Delta_p^{t+1}+\Delta_d^{t+1}\right)=0, \ \textrm{w.p.1.}
%\end{align}

{\bf Case 2)}: Suppose the stepsize is chosen according to \eqref{eq:step_size}.
%\begin{align}\label{eqStepSize}
%\sum_{t=1}^{\infty}\alpha^t=\infty, \ \lim_{t}\alpha^t=0.
%\end{align}
Then similar to Case 1,  we have the descent estimate given in \eqref{eq:descent}.

%\begin{align}\label{eqDescentExpectation2}
%&\mathbb{E}\left[(\Delta_p^{t}+\Delta_d^{t})-(\Delta_p^{t-1}+\Delta_d^{t-1})|z^t\right] \nonumber\\
%&\le
%(2\alpha^tp_0\|E\|^2\tau^{2}\hat{\sigma}_1^{2}-\hat{\gamma})\|\hx^{t+1}-x^t\|^2-\alpha^tp_0\|E\barhx^t-q\|^2
%+2(\alpha^t)^2p_0\|E\|^2\tau^{2}\hat{\sigma}^2_2\|q-Ex^t\|^2.
%\end{align}
The assumption $\alpha^t\to 0$ implies that there must exist an index $t_0$ such that for all $t>t_0$, the constants in front of $\|\hx^{t+1}-x^t\|^2$ and $\|E\barhx^{t+1}-q\|^2$ in \eqref{eq:descent} become negative. By applying the convergence theorem of non-negative almost supermartingale again, we conclude that
%
%again have $\mathbb{E}[\|\hx^{t+1}-x^t\|^2]\to 0$. By \eqref{eq:nice}, we have $\mathbb{E}[\|x^{t}-\bar{x}^t\|^2]\to 0$. Moreover, we have
\begin{align}
&\sum_{t=1}^{\infty}\|\hx^{t+1}-x^t\|^2<\infty, \ \mbox{w.p.1,}\label{eq:SumX}\\
&\sum_{t=1}^{\infty}\alpha^t\|E\barhx^{t+1}-q\|^2<\infty, \ \mbox{w.p.1}. \label{eq:SumE}
\end{align}
Eq.\ \eqref{eq:SumX} implies that
\begin{align}
\|\hx^{t+1}-x^t\|\to 0\quad \textrm{w.p.1.}\label{eq:primal_difference_zero_part2}
\end{align}
%{\textbf{Need more details here! }\color{black}Using the error bound condition in Lemma \ref{lm:eb}, we conclude that $\|\bar{x}^{t}-x^t\|\to 0$ w.p.1.}
while Eqs.\ \eqref{eq:SumE} and \eqref{eq:step_size} imply that
\begin{align}\label{eq:liminf}
\liminf_{t\to\infty}\|E\barhx^{t+1}-q\|^2=0, \ \mbox{w.p.1}.
\end{align}

To complete the proof, we show below that  in fact $\displaystyle\lim_{t\to\infty}\|E\barhx^{t+1}-q\|^2=0$ \ \mbox{w.p.1}.
Assume the contrary, so that there exists a $\delta>0$ such that
\begin{align}
\limsup_{t\to\infty}\|E\barhx^{t+1}-q\|^2=\delta>0. \label{eq:Contradiction}
\end{align}
Here and in what follows, all the statements hold in the almost sure sense, but we will omit the qualification ``w.p.1" for simplicity.
Using the Lipchitz continuity property of $\nabla d(y)$ (c.f., Lemma~\ref{lm:const-derivative}), we have
\begin{align}
\|E \barhx^{t+1}-q\|-\|E\barhx^t-q\|&\le \|(E\barhx^{t+1}-q)-(E\barhx^t-q)\|%\nonumber\\&
=\|\nabla d(\hy^{t+1})-\nabla d(\hy^t)\|\nonumber\\&
\le \frac{1}{\rho}\|\hy^{t+1}-\hy^t\|%\nonumber\\
%&= \frac{1}{\rho}\|\hy^{t+1}-y^{t}+y^{t}-\hy^t\|\nonumber\\&
{\le\frac{1}{\rho}\left(\|\hy^{t+1}-y^{t}\|+\|y^{t}-\hy^t\|\right)}\nonumber\\
&{\le\frac{1}{\rho}\left(\|\hy^{t+1}-y^{t}\|+\|y^{t-1}-\hy^t\|\right)}\nonumber\\
&=\frac{\alpha^{t}}{\rho}\|q-E x^{t}\|+{\frac{\alpha^{t-1}}{\rho}\|q-E x^{t-1}\|}. \label{eq:DifferentViolation}
\end{align}
%\textbf{Some index problem here? The red term is an upper bound for $\|y^{t-1}-\hy^t\|$, but ....}
%%Note that this series of inequalities is true no matter which block variable is updated in iteration $t$. If a primal block variable is updated, then $\bar{x}^{t+1}=\bar{x}^t$, and the above inequality is trivially true. If the dual variable is updated, then the last inequality becomes equality, but otherwise the overall inequality still holds true.
{Note that the second to the last inequality is true because if a primal variable is updated at iteration $t$, we have $\|y^{t}-\hy^t\|= \|y^{t-1}-\hy^t\|$; else we have $0=\|y^{t}-\hy^t\|\le\|y^{t-1}-\hy^t\|$.}
%Next we bound $\|q-E x^{t}\|$. By using a similar argument as in \eqref{eq:nice1} we can show that
%\begin{align}
%\|x^t-\barhx^{t}\|\le \tau\|\tilde{\nabla}L(x^t;\hat{y}^t)\|\le\tau\left(\tilde{\sigma}_1\|\hx^{t+1}-x^t\|+\tilde{\sigma}_2\alpha^{t-1}\|q-Ex^{t-1}\|\right).\label{eq:New}
%\end{align}
%where $\tilde{\sigma}_1>0$ and $\tilde{\sigma}_2>0$.
%Using \eqref{eq:New} and the fact that $\hy^{t+1}-y^t=\alpha^t(q-Ex^t)$, we obtain
%\begin{align}
%\|q-Ex^t\|&\le \|q- E\barhx^{t}\|+\|E(x^t-\barhx^{t})\|\le \|q- E\barhx^{t}\|+\|E\|\|x^t-\barhx^{t}\|\nonumber\\ \nonumber
%&\le \|q- E\barhx^{t}\|+\tau\|E\|(\tilde{\sigma}_1\|\hx^{t}-x^t\|+\tilde{\sigma}_2\alpha^{t-1}\|q-Ex^{t-1}\|).
%\end{align}
%It follows that when $t$ is large enough so that $\hat{\sigma}_2\alpha^t\le 1/2$, we have
%\begin{align}\label{eq:extra0}
%\|q-Ex^t\|\le 2\|q- E\barhx^{t+1}\|+2\tau\hat{\sigma}_1\|E\|\|\hx^{t+1}-x^t\|.
%\end{align}

{Now by \eqref{eq:liminf} and under the assumption \eqref{eq:Contradiction}, there must exist two infinite subsequences $\{t(n)\}$ and $\{t(p)\}$ such that
\begin{align}\label{eq:extra1}
&\|E\barhx^{t(n)}-q\|<\frac{\delta}{4}, \quad \|E\barhx^{t(n)+1}-q\|>\frac{\delta}{4}\nonumber\\
&\frac{\delta}{2}<\|E\barhx^{t(p)}-q\|, \quad \frac{\delta}{4}<\|E\barhx^{t}-q\|<\frac{\delta}{2},\ \ \forall \ t\in[t(n)+1,\ t(p)-1].
\end{align}
It follows from \eqref{eq:extra1} and \eqref{eq:DifferentViolation} that
\begin{align}
\frac{1}{4}\delta&<\|E\barhx^{t(p)}-q\|-\|E\barhx^{t(n)}-q\|\nonumber\\
&=\sum_{t=t(n)}^{t(p)-1}\left(\|E\barhx^{t+1}-q\|-\|E\barhx^{t}-q\|\right)\nonumber\\
%&\le \sum_{t=t(n)}^{t(p)-1}\frac{1}{\rho}\|\hy^{t+1}-\hy^{t}\|\nonumber\\
&\le \sum_{t=t(n)}^{t(p)-1}\left(\frac{\alpha^{t}}{\rho}\|q-E x^{t}\|+\frac{\alpha^{t-1}}{\rho}\|q-E x^{t-1}\|\right)\label{eq:BoundDifference}.
%&\le \sum_{t=t(n)}^{t(p)-1}\frac{\alpha^{t}}{\rho}\left(2\|q- E\barhx^{t{+1}}\|+2\tau\hat{\sigma}_1\|E\|\|\hx^{t+1}-x^t\|\right)\nonumber\\
%&\quad +\sum_{t=t(n)}^{t(p)-1}\frac{\alpha^{t-1}}{\rho}\left(2\|q- E\barhx^{t}\|+2\tau\hat{\sigma}_1\|E\|\|\hx^{t}-x^{t-1}\|\right)\label{eq:BoundDifference}.
\end{align}
%Due to the fact that $\|x^t-\hat{x}^{t+1}\|\to 0$, there must exist an index $t_1$ such that for all $t>t_1$, $\|x^{t-1}-\hat{x}^t\|\le\frac{\delta}{16\tau\|E\|\hat{\sigma}_1}$.

Due to the fact that $x^t$ lies in a compact set, there must exist a finite constant $\zeta>0$ such that $\|q-Ex^t\|\le \frac{\zeta}{16}\delta$ for all $t$. Combining this inequality with Eqs.\ \eqref{eq:extra1}-\eqref{eq:BoundDifference}, % and the fact that $\alpha^{t}\le\alpha^{t-1}$,
we conclude that for all $p$ and $n$ large enough,
\[
\frac{2\rho}{\zeta}<\sum_{t=t(n)}^{t(p)}{\alpha^{t-1}}.
\]
Since $\alpha^{t}\to0$, this further implies %it follows that when $p$ and $n$ are large enough there holds
\begin{align}
\frac{\rho}{\zeta}<\sum_{t=t(n)+1}^{t(p)}{\alpha^{t-1}},\quad \mbox{for large $p$ an $n$}.\label{eq:key}
\end{align}
From $\sum_{t}\alpha^t\|E\barhx^{t+1}-q\|^2<\infty$, we know that for any given $c>0$, there exist $p$ and $n$ large enough such that
\begin{align}\nonumber
\sum_{t=t(n)+1}^{t(p)}\alpha^{t-1}\|E\barhx^{t}-q\|^2\le c.
\end{align}
Since $\|E\barhx^{t}-q\|\ge \delta/4$ for $t\in[t(n)+1,t(p)-1]$ (see \eqref{eq:extra1}), it follows  that
\begin{align}\nonumber
\frac{\delta^2}{16}\sum_{t=t(n)+1}^{t(p)}\alpha^{t-1}< c.
\end{align}
%Notice that $\alpha^t\to0$ and $\|E\barhx^{t(n)}-q\|$ is bounded (by $\frac{\zeta}{8}\delta$), the above inequality further implies that for large enough $n$ and $p$,
%\[
%\frac{\delta^2}{16}\sum_{t=t(n)+1}^{t(p)-1}\alpha^{t-1}< c/2.
%\]
Let us set $c=\frac{\delta^2 \rho}{16 \zeta}$, then we have
\begin{align}\nonumber
\sum_{t=t(n)+1}^{t(p)}\alpha^{t-1}< \frac{\delta^2\rho}{16 \zeta}\times\frac{16}{\delta^2}=\frac{\rho}{\zeta},\ \ \mbox{ $\forall \ p,n$ large enough},
\end{align}}
which is a contradiction to \eqref{eq:key}. Hence, we must have
\begin{align}\nonumber
\limsup_{t\to\infty}\|E\barhx^t-q\|=0.
\end{align}
Combining with \eqref{eq:liminf}, we have
\begin{align}
\lim_{t\to\infty}\|E\barhx^t-q\|=0, \ \textrm{w.p.1}. \label{eq:dual_gap_zero_part2}
\end{align}
{Similar to the proof for part (1), using \eqref{eq:primal_difference_zero_part2} and \eqref{eq:dual_gap_zero_part2}, we conclude that with probability 1, every limit point of $\{x^{t}, y^t\}$ is a primal-dual optimal solution. }\QED

\begin{remark}
If either the objective function is strongly convex (i.e., $A$ is full column rank in Assumptions A(b)), or if $E$ is full column rank, then the augmented Lagrangian function $L(x;y)$ is strongly convex, implying that the error bound in Lemma~\ref{lm:eb} holds automatically and globally. In this case, Theorem~\ref{thm:main} holds  without Assumption A(b) nor the requirement of $X_k$ being polyhedral and  compact.
\end{remark}
\section{Unconstrained Convex Optimization}

In this section, we specialize the BSUM-M and the RBSUM-M methods to the unconstrained case. Since the linear coupling constraints are absent, the (randomized) BSUM-M reduces to the (randomize) BSUM algorithm, and stronger convergence results can be obtained.

\subsection{The BSUM Algorithm} \label{sub:BCD}
%When the coupling linear constraint $Ex=q$ is not present, the BSUM-M and RBSUM-M reduce to the conventional BSUM algorithm and the R-BSUM algorithm. More specifically,
Consider the following special case of problem \eqref{eq:1}
\begin{equation}\label{eq:bcd_problem}
\begin{array}{ll}
\mbox{minimize} &\displaystyle f(x):=g\left(x_1,\cdots, x_K\right)+\sum_{k=1}^{K}h_k(x_k)\\ [10pt]
\mbox{subject to}  & x_k\in X_k,\quad k=1,2,...,K.
\end{array}
\end{equation}
For each component $k$, define $u_k(\cdot; x)$: $X_k\mapsto \Re$ as a locally tight upper-bound for the smooth function $g(\cdot)$ at a given point $x\in X$. Below we will assume that $u_k(\cdot; x)$ satisfies Assumption B (with $\rho=0$).

The BSUM and R-BSUM algorithms are outlined in the following tables. Note that the BSUM-type algorithms described in this section are more general than the conventional BCD, in the sense that an approximation function $u_k(\cdot;\cdot)$ is used to update each component.
\begin{center}
\fbox{
\begin{minipage}{5.2in}
\smallskip
\centerline{\bf Block Successive Upper-bound Minimization (BSUM)}
\smallskip
At each iteration $r\ge 1$:
\begin{equation}\label{eq:BCD}
\begin{array}{l}\displaystyle
x_k^{r+1}={\rm arg}\!\min_{x_k\in X_k}u_k(x_k; w_k^{r+1})+h_k(x_k),~k=1,\cdots, K.\\[10pt]
\end{array}
\end{equation}
\end{minipage}
}
\end{center}

\begin{center}
\fbox{
\begin{minipage}{5.2in}
\smallskip
\centerline{\bf Randomized BSUM (R-BSUM)}
\smallskip
Select a probability vector $\{p_k\}_{k=1}^{K}$ such that $p_k>0$ and $\sum_{k=1}^{K}p_k=1$.\\
At each iteration $t\ge 1$, pick an index $k\in \{1,\cdots, K\}$, with probability $p_k$
\begin{equation}\label{eq:R-BCD}
\begin{array}{l}\displaystyle
\quad x_k^{t+1}={\rm arg}\!\min_{x_k\in X_k}u_k(x_k;  x^t)+h_k(x_k),\\[10pt]
\quad x_j^{t+1}=x_j^t,\ \forall~j\ne k. \\[10pt]
\end{array}
\end{equation}
\end{minipage}
}
\end{center}
%
%Let us make the following assumption on the approximation function $u_k(\cdot;\cdot)$.
%\pn {\bf Assumption C.}
%\begin{itemize}
%\item [(a)]  $u_k(x_k; x) = g(x), \quad \forall\; x\in {X}, \ \forall\; k.$
%\item [(b)] $u_k(z_k; x)\geq g\left([x_{-k}, z_k] \right),\quad\; \forall\; z_k \in {X}_k, \ \forall\; x \in{X}, \ \forall\; k.$
%\item [(c)] $\nabla u_k(x_k;x)= \nabla_{x_k}(g(x)), \quad \forall\; k, \ \forall\; x\in X$
%\item [(d)] For given $x\in X$, $u_k(z_k; x)$ is continuous and strongly convex in $X_k$, that is
%$$u_k(z_k; x)\ge u_k(\hat{z}_k; x)+\langle\nabla u_k(\hat{z}_k; x),z_k-\hat{z}_k\rangle+\frac{\gamma_k}{2}\|z_k-\hat{z}_k\|^2,\ \forall~z_k, \ \hat{z}_k\in X_k, \ \forall~x\in X$$
%where $\gamma_k$ is independent of the choice of $x$.
%\item [(e)] For given $x\in X$, $u_k(z_k; x)$ has Lipchitz continuous gradient, that is
%\begin{align}
%\|\nabla u_k(z_k; x)-\nabla u_k(\hat{z}_k; x)\|\le L_k\|z_k-\hat{z}_k\|,\ \forall\ \hat{z}_k,\ z_k\in X_k, \ \forall
%\ k, \ \forall~x\in X,
%\end{align}
%where $L_k>0$ is some constant.
%\end{itemize}
%Clearly Assumption C is a specialization of Assumption B for problem \eqref{eq:bcd_problem}.

\subsection{Linear Convergence of the BSUM Algorithm} \label{sub:BCDLinear}

In this section, we show that under similar assumptions given in Section \ref{sub:assumptions}, both BSUM and R-BSUM converge linearly. For the BSUM algorithm, define the optimality gap as $\Delta^r:=f(x^{r+1})-f({x}^*)$, where ${x}^*\in X^*$ is an optimal solution. Similarly, for the R-BSUM algorithm, define $\Delta^t:=f(x^{t+1})-f({x}^*)$.

We first note that these algorithms indeed converge. This is a consequence of Theorem \ref{thm:main} (just ignore the linear constraints).

%\begin{corollary}\label{cor:bcd_sufficient_decrease}
%Assume Assumption B holds true. Then
%\begin{align}
%&f(x^{r+1})-f(x^{r})=\Delta^r-\Delta^{r-1}\le  -\gamma\|x^{r+1}-x^{r}\|^2\label{eq:bcd_sufficient_decrease}\\
%&\mathbb{E}\left[f(x^{t+1})-f(x^{t})\mid x^t\right]=\mathbb{E}\left[\Delta^t-\Delta^{t-1}\mid x^t\right]\le -\hat{\gamma}\|\hat{x}^{t+1}-x^{t}\|^2\label{eq:sbcd_sufficient_decrease}
%\end{align}
%for some $\gamma=\min_k\gamma_k>0$, $\hat{\gamma}=\min_k p_k\gamma_k>0$.
%\end{corollary}
%
%\begin{corollary} \label{cor:estimate_bcd}
%{Suppose assumptions A(b) and Assumption B hold. Then
%\begin{enumerate}
%\item For the iterates $\{x^r\}$ generated by the BCD, there exists some constant
%$\sigma>0$ such that
%\begin{equation}\label{eq:size_prox_bcd}
%\|\tilde\nabla f(x^r)\|\le \sigma\|x^{r+1}-x^r\|,\quad \mbox{for all } r\ge 1.
%\end{equation}
%\item For the iterates $\{x^t\}$ generated by the R-BCD, there exists some constant
%$\hat{\sigma}>0$ such that
%\begin{equation}\label{eq:size_prox_sbcd}
%\|\tilde\nabla f(x^t)\|\le \hat{\sigma}\|\hx^{t+1}-x^t\|,\quad\mbox{for all } t\ge 1.
%\end{equation}
%\end{enumerate}
%}
%\end{corollary}

\begin{corollary}\label{cor:bcd}
Suppose Assumptions A(a) and B hold. Then we have the following:
\begin{enumerate}
\item For the BSUM, the sequences
$\{\Delta^r\}$ and $\{\|x^r-x^{r+1}\|\}$ both converge to zero. Further, every limit point of $\{x^r\}$ is an optimal solution for problem \eqref{eq:bcd_problem}.
\item For the R-BSUM, the sequences
$\{\Delta^t\}$ and $\{\|x^t-x^{t+1}\|\}$ both converge to zero w.p.1. Further, every limit point of $\{x^t\}$ is an optimal solution for problem \eqref{eq:bcd_problem} w.p.1.
\end{enumerate}
\end{corollary}

\begin{remark}
The conditions used in Corollary \ref{cor:bcd} are slightly stronger than those for the original BSUM algorithm \cite{Razaviyayn12SUM}. In particular, here we require that the per-block upper-bound function $u_k(v_k; x)$ is strongly convex with respect to $v_k$, while in \cite[Theorem 2(a)]{Razaviyayn12SUM}, it is only assumed that $u_k(v_k; x)$ is quasi-convex, and that problem $\min_{v_k\in X_k} u_k(v_k; x)$ has a unique optimal solution. The per-block strong convexity is needed here to show part $(2)$ of Corollary \ref{cor:bcd}.
\end{remark}

\begin{remark}
Different from the proof of Theorem \ref{thm:main}, Corollary \ref{cor:bcd} does not require Assumption {\rm A(b)-(c)}. Such assumptions are needed in Theorem \ref{thm:main} to invoke the error bound property \eqref{eq:bds}, which in turn is used to establish the key descent property of the combined primal and dual gaps $($cf.\ \eqref{eq:descent}$)$. In contrast, the analysis of BSUM/R-BSUM only involves the primal gaps, whose descent is guaranteed by the algorithms. The error bound, however, is needed below to establish linear convergence.
\end{remark}

To show linear convergence of these algorithms, we need an additional result that bounds the size of the optimality gap.

\begin{lemma}\label{lm:cost-to-go}
We have the following estimate of the optimality gaps.
\begin{enumerate}
\item For the BSUM, suppose Assumption A(a) and Assumption B hold. Then there exist positive scalars $\zeta$ and $\zeta^{'}$
$($independent of $y^r)$ such that
\begin{align}\label{eq:p-gap-bcd}
\Delta^r\le \zeta\|x^{r+1}-x^r\|^2+\zeta^{'}\|x^r-\bar{x}^r\|^2,\quad\mbox{for all $r\ge1$.}
\end{align}

\item For the R-BSUM, suppose Assumptions A and B hold. Additionally, assume that $h_k(x_k)=\lambda_k\|x_k\|_1$ for any $\lambda_k\ge 0$, and that $C_k$ is full row rank for each $k$. Then there exists a finite $t_0>0$ and  a positive scalar $\hat{\zeta}^{'}$ $($independent of $y^r)$ such that
\begin{equation}\label{eq:p-gap-sbcd}
\mathbb{E}\left[\Delta^t\mid x^{t}\right]\le \hat{\zeta}^{'}\|x^t-\bar{x}^t\|^2,\quad\mbox{for all $t\ge t_0$.}
\end{equation}
\end{enumerate}
\end{lemma}

\proof We only show part (2) of the proof. Part (1) of the proof is much simpler, and can be found in  \cite[Lemma 3.1]{HongLuo2012ADMM}.

From the mean value theorem, for any $k$ there exists some
$\tilde{x}^{t}$ in the line segment joining $x^{t}$ and $\bar{x}^t$
such that
\[g(x^{t})-g(\bar{x}^t)=\langle \nabla g(\tilde{x}^{t}), x^{t}-\bar{x}^t\rangle. \]

The conditional expected value of $\Delta^t$ can be bounded above by
\begin{align}
\mathbb{E}[\Delta^t\mid x^t]&=\sum_{k=1}^{K}p_k\left(f(x^t_{-k}, \hx^{t+1}_k)-f(\bar{x}^t)\right)\nonumber\\
&\le\sum_{k=1}^{K}p_k\left(f(x^t)-f(\bar{x}^t)\right)\nonumber\\
&=\left\langle\nabla g(\tilde{x}^t), x^t-\bar{x}^t\right\rangle+h(x^t)-h(\bar{x}^t) \nonumber\\
&=\left\langle\nabla g(\tilde{x}^t)-\nabla g(\bar{x}^t), x^t-\bar{x}^t\right\rangle+\left\langle\nabla g(\bar{x}^t), x^t-\bar{x}^t\right\rangle+h(x^t)-h(\bar{x}^t)\nonumber\\
&\le  L \|\tilde{x}^t-\bar{x}^t\|\| x^t-\bar{x}^t\|+\left\langle\nabla g(\bar{x}^t), x^t-\bar{x}^t\right\rangle+h(x^t)-h(\bar{x}^t)\nonumber\\
&\le  L \| x^t-\bar{x}^t\|^2+{\color{black}\left\langle\nabla g(\bar{x}^t), x^t-\bar{x}^t\right\rangle+h(x^t)-h(\bar{x}^t)}\nonumber
\end{align}
where the last inequality comes from the fact that $\tilde{x}^{t}$ lies in the line segment joining $x^{t}$ and $\bar{x}^t$. In the following, we will show that when $t$ is large enough, with probability 1 we have  $\left\langle\nabla g(\bar{x}^t), x^t-\bar{x}^t\right\rangle+h(x^t)-h(\bar{x}^t)=0$.

We first observe that $\bar{x}^t$ satisfies
\begin{align}\noindent\nonumber
\langle\nabla g(\bar{x}^t)+\partial h(\bar{x}^t), x-\bar{x}^t\rangle\ge 0,\ \forall~x\in X.
\end{align}
This implies that
\begin{align}\noindent\nonumber
h(x)-h(\bar{x}^t)+\langle\nabla g(\bar{x}^t), x-\bar{x}^t\rangle\ge 0, \ \forall~x\in X.
\end{align}
Assumption A(b) implies that  $\nabla g(x^*)$  takes the same value for any $x^*\in X^*$. Let us denote
\begin{equation}\label{eq:a}
\nabla g(x^*)=a^*,\quad \forall\ x^*\in X^*.
\end{equation}
Then for any two optimal solutions $x^*, x^{\infty}\in X^*$, we have
\begin{align}
h(x^*)-h(x^{\infty})+\langle\nabla g({x}^\infty), x^*-{x}^{\infty}\rangle\ge 0,\nonumber\\
h(x^\infty)-h(x^{*})+\langle\nabla g({x}^*), x^{\infty}-{x}^{*}\rangle\ge 0\nonumber.
\end{align}
Using the fact that $\nabla g({x}^*)=\nabla g({x}^\infty)$, we conclude
\begin{align}
h(x^\infty)+\langle\nabla g(x^{\infty}),x^{\infty}\rangle&=h(x^*)+\langle\nabla g(x^{\infty}),x^{*}\rangle\nonumber\\
&=h(x^*)+\langle\nabla g(x^{*}),x^{*}\rangle, \ \forall~x^*, x^{\infty}\in X^*.\label{eq:optimal_equivalence}
\end{align}

The main part of the proof is to show that for $t$ large enough, there exists an $x^{\infty}\in X^*$ such that
\begin{align}
\left\langle\nabla g({x}^\infty), x^t-{x}^\infty\right\rangle+h(x^t)-h({x}^\infty)=0. \label{eq:target}
\end{align}
If the above relation is true, then utilizing \eqref{eq:optimal_equivalence}, we can conclude that $\left\langle\nabla g(\bar{x}^t), x^t-\bar{x}^t\right\rangle+h(x^t)-h(\bar{x}^t)=0$.

For any given block $k$ and any given iteration index $t>0$, let $t(k)$ denote the last iteration such that $x^t_k$ has been updated, i.e.,
$t(k):=\max\{j \mid j<t, x_k^{j}\ne x_k^t\}.$ Then according to the way that $x_k$ is updated, we have
\begin{align}
x^t_k&=x^{t(k)+1}_k=\prox\left[x^{t(k)+1}_k-\nabla u_k\left(x^{t(k)+1}_k; x^{t(k)}\right)\right]\nonumber\\
&=\prox\left[x^{t-1}_k-\nabla_kg\left(x^{t-1}\right)+e^t_k\right] \label{eq:sbcd_prox}
\end{align}
where we have defined $e^t_k$ as
\begin{align}\nonumber
e^t_k:&=\nabla_kg\left(x^{t-1}\right)-\nabla u_k\left(x^{t(k)+1}_k; x^{t(k)}\right)+x^{t(k)+1}_k-x_k^{t-1}.
\end{align}
Clearly the norm of the error term $e^t_k$ is bounded by
\begin{align}
\|e^t_k\|&=\left\|\nabla_kg\left(x^{t-1}\right)-\nabla u_k\left(x^{t(k)+1}_k; x^{t(k)}\right)+x^{t(k)+1}_k-x_k^{t-1}\right\|\nonumber\\
&=\left\|\nabla u_k\left(x^{t(k)}_k; x^{t(k)}\right)-\nabla u_k\left(x^{t(k)+1}_k; x^{t(k)}\right)+\nabla_kg\left(x^{t-1}\right)-\nabla u_k\left(x^{t(k)}_k; x^{t(k)}\right)+x^{t}_k-x_k^{t-1}\right\|\nonumber\\
&\le L_k\|x^{t(k)}_k-x^{t(k)+1}_k\|+\left\|\nabla_kg\left(x^{t-1}\right)-\nabla u_k\left(x^{t(k)}_k; x^{t(k)}\right)\right\|+\|x^{t}_k-x_k^{t-1}\|\nonumber\\
&=L_k\|x^{t(k)}_k-x^{t(k)+1}_k\|+\left\|\nabla_kg\left(x^{t-1}\right)-\nabla_kg\left(x^{t(k)}\right)\right\|
+\|x^{t}_k-x_k^{t-1}\|.\nonumber
\end{align}
The fact that $p_k$ is bounded away from $0$ for all $k$ implies that if $t\to\infty$, then $t(k)\to \infty$ for all $k$ w.p.1. {\color{black}Thus, using the results in the second part of Corollary \ref{cor:bcd}}, we have that both $x^{t(k)}$ and $x^{t-1}$ converge to the set of $X^*$ w.p.1 (though they may have different limit points), and {\color{black}that $\|x^{t(k)}_k-x^{t(k)+1}_k\|\to 0$ and $\|x^{t}_k-x^{t-1}_k\|\to 0$ w.p.1}. These results imply that
\begin{align}\nonumber
\lim_{t\to\infty}\|e^t_k\|&=\lim_{t\to\infty}L_k\|x^{t(k)}_k-x^{t(k)+1}_k\|+\left\|\nabla_kg\left(x^{t-1}\right)-\nabla_kg\left(x^{t(k)}\right)\right\|+\|x^{t}_k-x_k^{t-1}\|\nonumber\\
&=0+\|a_k^*-a_k^*\|+0=0,
\end{align}
where $a^*$ is defined by \eqref{eq:a}.

Next we show \eqref{eq:target}. To proceed, we need a few new definitions. Let $C_k[j]$ denote the $j$th row of the matrix $C_k$, let $c_k[j]$ denote the $j$th element of the vector $c_k$. Define $\cI_k$ as the set of indices contained in $x_k$; $\cJ^t_k$ as the set of indices of active constraints for block $k$ at iteration $t$: $\cJ^t_k:=\{j: C_k[j] x^t_k=c_k[j]\}$.  Eq.\ \eqref{eq:sbcd_prox} and the fact $h_k(x_k)=\lambda_k\|x_k\|_1$ imply that the optimality condition for block variable $k$ at iteration $t$ is given by
%\begin{align}\label{eq:per_step_opt_1}
%\left\{
%\begin{array}{l}
%x^t_k-x^{t-1}_k+\nabla_kg(x^{t-1})-e^t_k+\lambda_k\partial\|x^t_k\|_1+C^T_k\mu^t_k=0\\
%C_k[j]x^t_k=c_k[j],\ \forall~i\in \cJ^t_k, \ \mu^t_k[j]=0,\ \forall~j\notin\cJ^t_k,\ \mu^t_k\ge 0
%\end{array}
%\right.
%\end{align}
%where $\mu^t_k$ is the Lagrangian multiplier associated with the constraint $C_k x^t_k\le c_k$. Using the definition of the subgradient of $\|\cdot\|_1$, eq. \eqref{eq:per_step_opt_1} is equivalent to the following
\begin{align}\label{eq:per_step_opt}
\left\{
\begin{array}{l}
-\lambda_k\le\left[x^t_k-x^{t-1}_k+\nabla_kg(x^{t-1})+C^T_k\mu^t_k-e^t_k\right]_i\le \lambda_k, \ \forall~i\in \cI_k\\
x^t_k[i]\ge 0\ \textrm{if} \ [x^t_k-x^{t-1}_k+\nabla_kg(x^{t-1})+C^T_k\mu^t_k-e^t_k]_i=-\lambda_k\\
x^t_k[i]\le 0\ \textrm{if} \ [x^t_k-x^{t-1}_k+\nabla_kg(x^{t-1})+C^T_k\mu^t_k-e^t_k]_i=\lambda_k\\
x^t_k[i]= 0\ \textrm{if} \ -\lambda_k<[x^t_k-x^{t-1}_k+\nabla_kg(x^{t-1})+C^T_k\mu^t_k-e^t_k]_i<\lambda_k\\
\ C_k[j]x^t_k=c_k[j],\ \forall~i\in \cJ^t_k, \ \mu^t_k[j]=0,\ \forall~j\notin\cJ^t_k,\ \mu^t_k\ge 0
\end{array}
\right.
\end{align}
where the notation $[\cdot]_i$ denotes the $i$th element of a vector.

%Similarly, at global optimality, the $k$th block variable, say $x^*_k$, must satisfy
%%\begin{align}
%%\left\{
%%\begin{array}{l}
%%\nabla_kg(x^*)+C^T_k\mu^*_k=-\lambda_k\partial\|x^*_k\|_1\\
%%C_k[i]x^*_k=c_k[i],\ \forall~i\in \cJ^*_k, \ \mu^*_k[i]=0,\ \forall~i\notin\cJ^*_k,\ \mu^*_k\ge 0
%%\end{array}
%%\right.
%%\end{align}
%%Equivalently, using the definition of $\partial\|x^*_k\|_1$, we have
%\begin{align}
%\left\{
%\begin{array}{l}
%-\lambda_k\le\left[\nabla_kg(x^*)+C^T_k\mu^*_k\right]_i\le \lambda_k, \ \forall~i\in \cI_k\\
%x^*_k[i]\ge 0\ \textrm{if} \ [\nabla_kg(x^*)+C^T_k\mu^*_k]_i=-\lambda_k\\
%x^*_k[i]\le 0\ \textrm{if} \ [\nabla_kg(x^*)+C^T_k\mu^*_k]_i=\lambda_k\\
%x^*_k[i]= 0\ \textrm{if} \ -\lambda_k<[\nabla_kg(x^*)+C^T_k\mu^*_k]_i<\lambda_k\\
%\ C_k[i]x^*_k=c_k[i],\ \forall~i\in \cJ^*_k, \ \mu^*_k[i]=0,\ \forall~i\notin\cJ^*_k,\ \mu^*_k\ge 0
%\end{array}
%\right..
%\end{align}
%where $\mu^*_k$, $\cJ^*_k$ are defined similarly as above.
%
%For a given solution $x^*\in X^*$, define the index sets $\{\cI^+_k(x^*)\}_{k}$, $\{\cI^-_k(x^*)\}_{k}$ and $\{\cI^=_k(x^*)\}_{k}$ as follows
%\begin{align}\label{eq:defI}
%\left\{
%\begin{array}{l}
%\cI^{+}_k(x^*)=\left\{i: [\nabla_kg(x^*)+C^T_k\mu^*_k]_i=-\lambda_k, i\in \cI_k\right\}\\
%\cI^{-}_k(x^*)=\left\{i: [\nabla_kg(x^*)+C^T_k\mu^*_k]_i=\lambda_k, i\in \cI_k\right\}\\
%\cI^{=}_k(x^*)=\left\{i: \left|[\nabla_kg(x^*)+C^T_k\mu^*_k]_i\right|<\lambda_k, i\in \cI_k\right\}\\
%\end{array}
%\right..
%\end{align}

Define $\cJ^t$ as the index set of active constraints at iteration $t$:
$$\cJ^t:=\bigcup_{k=1}^{K} \cJ^t_k=\bigcup_{k=1}^{K} \{j: j\in \cI_k, C_k[j] x_k^t=c_k[j], k=1,\cdots, K\}.$$
Since there are only a finite number of distinct choices of $\cJ^t$, it follows that there exists some $\cJ^\infty$ such that $\cJ^t=\cJ^\infty$ for an infinite number of $t$.
%For any such index set $\cJ$, define the set of iteration indices
%$$\cT^\cJ:=\{t\mid t\in\{1, 2,\cdots\}, \cJ^t=\cJ\}.$$
%Then for any given $\cJ$ such that $\cT^\cJ$ is infinite, any $x^{s}$ and $x^t$ with $t,s\in \cT^{\cJ}$ must have the same set of active constraints.
Due to the compactness of $X_k$, the full rankness of $C^T_k$, and the fact that $e^t\to 0$, it follows that for sufficiently large $t$, $x^t_k$ and $\mu^t_k$ are both bounded. By further passing to a subsequence $\cT$ if necessary, we can assume that
\begin{equation}\label{eq:assumption}
\begin{array}{l}
\displaystyle \lim_{t\in\cT,t\to\infty}x_k^t= \lim_{t\in\cT,t\to\infty}x_k^{t-1}=x_k^{\infty},
 \quad \lim_{t\in\cT,t\to\infty}\mu_k^t=\mu_k^{\infty},\\ \displaystyle
\lim_{t\in\cT,t\to\infty}e_k^t= 0,\quad \lim_{t\in\cT,t\to\infty} (x_k^t-x^{t-1}_k)= 0, \quad \forall\ k.
\end{array}
\end{equation}
 Taking limit $t\to\infty$ along $\cT$, we obtain from \eqref{eq:per_step_opt} the following for all $k=1,\cdots, K$:
\begin{align}\label{eq:optimality_infty}
\left\{
\begin{array}{l}
-\lambda_k\le\left[\nabla_kg(x^{\infty})+C^T_k\mu^{\infty}_k\right]_i\le \lambda_k, \ \forall~i\in \cI_k\\
x^{\infty}_k[i]\ge 0\ \textrm{if} \ [\nabla_kg(x^{\infty})+C^T_k\mu^{\infty}_k]_i=-\lambda_k\\
x^{\infty}_k[i]\le 0\ \textrm{if} \ [\nabla_kg(x^{\infty})+C^T_k\mu^{\infty}_k]_i=\lambda_k\\
x^{\infty}_k[i]= 0\ \textrm{if} \ -\lambda_k<[\nabla_kg(x^{\infty})+C^T_k\mu^{\infty}_k]_i<\lambda_k\\
\ C_k[j]x^{\infty}_k=c_k[j],\ \forall~j\in \cJ^{\infty}_k, \ \mu^{\infty}_k[j]=0,\ \forall~j\notin\cJ_k,\ \mu^{\infty}_k\ge 0
\end{array}
\right.
\end{align}
where $\mu^{\infty}_k$, $\cJ^{\infty}_{k}$ are the corresponding components of $\mu^\infty$ and $\cJ^\infty$ respectively.

%Further, from \eqref{eq:per_step_opt}, we have
%\begin{align}\label{eq:per_step_opt_2}
%\left\{
%\begin{array}{l}
%x^t_k[i]\ge 0\ \textrm{if}\ [x^{t-1}_k-\nabla_kg(x^{t-1})-C^T_k\mu^t_k+e^t_k]_i\ge \lambda_k\\
%x^t_k[i]\le 0\ \textrm{if} \ [x^{t-1}_k-\nabla_kg(x^{t-1})-C^T_k\mu^t_k+e^t_k]_i\le-\lambda_k\\
%x^t_k[i]= 0\ \textrm{if} \ -\lambda_k\le[x^{t-1}_k-\nabla_kg(x^{t-1})-C^T_k\mu^t_k+e^t_k]_i\le\lambda_k
%\end{array}
%\right.
%\end{align}

In the following, we compare the two systems \eqref{eq:per_step_opt}, \eqref{eq:optimality_infty} and show that when $t$ becomes large enough, $x^t$ and $x^{\infty}$ will have the same sign pattern. To make this statement precise,  let us define three index sets below %$\{\cI^+_k(x^\infty)\}_{k}$, $\{\cI^-_k(x^\infty)\}_{k}$ and $\{\cI^=_k(x^\infty)\}_{k}$ as follows
%\begin{align}\label{eq:defI}
%\left\{
%\begin{array}{l}
%\cI^{+}_k(x^\infty)=\left\{i: [\nabla_kg(x^\infty)+C^T_k\mu^\infty_k]_i=-\lambda_k, i\in \cI_k\right\}\\
%\cI^{-}_k(x^\infty)=\left\{i: [\nabla_kg(x^\infty)+C^T_k\mu^\infty_k]_i=\lambda_k, i\in \cI_k\right\}\\
%\cI^{=}_k(x^\infty)=\left\{i: \left|[\nabla_kg(x^\infty)+C^T_k\mu^\infty_k]_i\right|<\lambda_k, i\in \cI_k\right\}\\
%\end{array}
%\right..
%\end{align}
\begin{align}\label{eq:defI}
\left\{
\begin{array}{l}
\cI^{+}_k =\left\{i: [\nabla_kg(x^\infty)+C^T_k\mu^\infty_k]_i=-\lambda_k, i\in \cI_k\right\}\\
\cI^{-}_k =\left\{i: [\nabla_kg(x^\infty)+C^T_k\mu^\infty_k]_i=\lambda_k, i\in \cI_k\right\}\\
\cI^{=}_k=\left\{i: \left|[\nabla_kg(x^\infty)+C^T_k\mu^\infty_k]_i\right|<\lambda_k, i\in \cI_k\right\}\\
\end{array}
\right..
\end{align}
We claim that for $t$ large enough, the following {\it identifiability condition} is true for each $k=1,\cdots, K$:
\begin{equation}\label{eq:identifiability}
x_k^t[i]\ge0, \ \forall\; i\in \cI_k^{+}, \quad
x_k^t[i]\le 0, \ \forall\; i\in \cI_k^{-}, \quad
x_k^t[i]=0, \ \forall\; i\in \cI_k^{=}.
\end{equation}

Suppose $i\in \cI_k^{+}$, then we have $[\nabla_kg(x^{\infty})+C^T_k\mu^{\infty}_k]_i=-\lambda_k$.
{The fact that $x_k^{\infty}[i]\ge 0$ implies
%\begin{align}
\[
[x^{\infty}_k-\nabla_kg(x^{\infty})-C^T_k\mu^{\infty}_k]_i\ge \lambda_k.\label{eq:limit_plus}
\]
%\end{align}
By \eqref{eq:assumption}, we have
\[
\lim_{t\in\cT,t\to\infty}\left[x^{t-1}_k-\nabla_kg(x^{t-1})-C^T_k\mu^t_k+e^t_k\right]_i=
[x^{\infty}_k-\nabla_kg(x^{\infty})-C^T_k\mu^{\infty}_k]_i\ge \lambda_k,
\]
%
%Using the following bound
%\begin{align}
%&\|(x^{\infty}_k-\nabla_kg(x^{\infty})-C^T_k\mu^{\infty}_k)-(x^{t-1}_k-\nabla_kg(x^{t-1})-C^T_k\mu^t_k-e^t_k)\|\nonumber\\
%&\le \|x^{\infty}_k-x^{t-1}_k\|+L\|x_k^{\infty}-x_k^{t-1}\|+\|C^T_k(\mu^{\infty}_k-\mu^t_k)\|+\|e^t_k\|\nonumber
%\end{align}
%and the fact that $x^{t-1}_k\to x^{\infty}_k$, $\mu^t_k\to \mu^{\infty}_k$, $e^{t}_k\to 0$ and \eqref{eq:limit_plus},
which further implies that there exists some $t_1$ such that for all $t>t_1$
\[
\left[x^{t-1}_k-\nabla_kg(x^{t-1})-C^T_k\mu^t_k+e^t_k\right]_i\ge\frac{1}{2}\lambda_k> -\lambda_k.
\]
If $x^t_k[i]<0$, then it follows from the above inequality that
\[
\left[x^t_k-x^{t-1}_k+\nabla_kg(x^{t-1})+C^T_k\mu^t_k-e^t_k\right]_i< \lambda_k
\]
which by \eqref{eq:per_step_opt} would imply $x^t_k[i]\ge 0$, a contradiction. Thus, we must have $x^t_k[i]\ge 0$ for all $t> t_1$.}

Using a similar argument, we can show that there exists some $t_2$ such that $x^t_k[i]\le 0$ for any  $i\in \cI_k^{-}$ and for all $t>t_2$.
For any $i\in \cI_k^{=}$, then there holds
\[
x^\infty_k[i]=0,\ \ \left|[\nabla_kg(x^{\infty})+C^T_k\mu^{\infty}_k]_i\right|<\lambda_k.
\]
It follows from \eqref{eq:assumption} that
\[
\lim_{t\in\cT,t\to\infty}\left|\left[-\nabla_kg(x^{r-1})+x_k^{r-1}-C^T_k\mu^t_k+e^t_k\right]_i\right|
=\left|[\nabla_kg(x^{\infty})+C^T_k\mu^{\infty}_k]_i\right|<\lambda_k,
\]
which further implies that there exists some $t_3$ such that
\[
\left|\left[-\nabla_kg(x^{r-1})+x_k^{r-1}-C^T_k\mu^t_k+e^t_k\right]_i\right|<\lambda_k, \quad \forall\ t\ge t_3.
\]
We prove by contradiction that $x_k^t[i]=0$ for all $t>t_3$. Specifically, if $x^t_k[i]>0$, then the above inequality implies
\[
\left[x_k^t-x_k^{r-1}+\nabla_kg(x^{r-1})+C^T_k\mu^t_k+e^t_k\right]_i>x^t_k[i]-\lambda_k>-\lambda_k,
\]
which by \eqref{eq:per_step_opt} further implies that $x^t_k[i]\le 0$, a contradiction. Similarly, $x^t_k[i]$ cannot be negative either. Thus, we have $x^t_k[i]=0$ for all $i\in\cI^{=}_k$ and all $t\ge t_3$.
This completes the proof of the identifiability property \eqref{eq:identifiability} for all $k$.

Now we are ready to show \eqref{eq:target}. Assume that $t$ is large enough such that the identifiability condition \eqref{eq:identifiability} is true. Suppose $i\in \cI^+_k$, we have
\begin{eqnarray*}
\left[\nabla_kg(x^{\infty})\right]_i[x^{t}_k-x^{\infty}_k]_i
&=&\left(-\lambda_k-[C^T_k\mu^{\infty}_k]_i\right)[x^{t}_k-x^{\infty}_k]_i \\
&=&\left(-\lambda_k-C^T_k[i]\mu^{\infty}_k\right)[x^{t}_k-x^{\infty}_k]_i\\
&=&-\lambda_k[x^{t}_k-x^{\infty}_k]_i- \langle\mu^{\infty}_k, \left(C^T_k[i]\right)^{T}([x^{t}_k-x^{\infty}_k]_i)\rangle.
\end{eqnarray*}
Similarly, we have
\begin{eqnarray*}
\left[\nabla_kg(x^{\infty})\right]_i[x^{t}_k-x^{\infty}_k]_i&=&\lambda_k[x^{t}_k-x^{\infty}_k]_i- \langle\mu^{\infty}_k, \left(C^T_k[i]\right)^T([x^{t}_k-x^{\infty}_k]_i)\rangle,\quad  \forall\ i\in \cI^-_k,\\
\left[\nabla_kg(x^{\infty})\right]_i[x^{t}_k-x^{\infty}_k]_i&=&- \langle\mu^{\infty}_k, \left(C^T_k[i]\right)^T([x^{t}_k-x^{\infty}_k]_i)\rangle=0,\quad \forall\ i\in \cI^=_k.
\end{eqnarray*}
Using the above relations, we obtain for $t$  large enough,
\begin{eqnarray*}
\langle\nabla g(x^{\infty}), x^t-x^{\infty}\rangle+h(x^t)-h(x^{\infty})
&=&\sum_{k=1}^{K}\sum_{i\in \cI_k}\left(\left[\nabla_kg(x^{\infty})\right]_i[x^{t}_k-x^{\infty}_k]_i
+\lambda_k\left(|x^t_k[i]|-|x^{\infty}_k[i]|\right)\right)\nonumber\\
&=&\sum_{k=1}^{K}\sum_{i\in\cI^{+}}-\lambda_k[x^{t}_k-x^{\infty}_k]_i+
\sum_{k=1}^{K}\sum_{i\in\cI^{-}}\lambda_k[x^{t}_k-x^{\infty}_k]_i\nonumber\\
&&\quad-\sum_{k=1}^{K}\langle\mu^{\infty}_k,C_k(x^t_k-x^{\infty}_k)\rangle+\lambda_k\left(|x^t_k[i]|-|x^{\infty}_k[i]|\right)\nonumber\\
&=&\sum_{k=1}^{K}\sum_{i\in\cI^{+}}-\lambda_k[x^{t}_k-x^{\infty}_k]_i+
\sum_{k=1}^{K}\sum_{i\in\cI^{-}}\lambda_k[x^{t}_k-x^{\infty}_k]_i\nonumber\\
&&\quad +\sum_{k=1}^{K}\sum_{i\in\cI^{+}}\lambda_k[x^{t}_k-x^{\infty}_k]_i+
\sum_{k=1}^{K}\sum_{i\in\cI^{-}}-\lambda_k[x^{t}_k-x^{\infty}_k]_i\nonumber\\
&=&0.
%&=\sum_{k=1}^{K}\sum_{i\in\cI^{+}(x^{\infty})\bigcup\cI^{-}(x^{\infty})}
%\left(\left[\nabla_kg(x^{\infty})\right]_i[x^{t}_k-x^{\infty}_k]_i+\left(|x^r_k[i]|-|x^{\infty}_k[i]|\right)\right)
\end{eqnarray*}
Note that the second to the last equality is due to $\cJ^t=\cJ^\infty$ for all $t\in \cT$, which implies
\[C_k[j]x^{\infty}_k=c_k,\ \forall~j\in \cJ_k, \ \mu^{\infty}_k[j]=0,\ \forall~j\notin\cJ_k,\]
so that $\sum_{k=1}^{K}\langle\mu^{\infty}_k,C_k(x^t_k-x^{\infty}_k)\rangle=0$.
Using \eqref{eq:optimal_equivalence}, we obtain the desired result. This completes the proof for part (2) of the lemma.
\QED

We remark that the identifiability property \eqref{eq:identifiability} has been observed numerically by Richt\'{a}rik and Tak\'{a}\v{c} in \cite[Section 6.1.7]{richtarik12}  when using a randomized block coordinate descent method to solve a certain $\ell_1$-minimization problem. Here in the proof of Lemma \ref{lm:p-descent}, we have established this property theoretically.

Next we use Lemma \ref{lm:p-descent} to show that both the BSUM and R-BSUM converge linearly.
\begin{theorem}\label{thm:main_linear_bcd}
{Suppose Assumptions A and B hold. Then we have the following:
\begin{enumerate}
\item  For the BSUM algorithm, the sequence $\{\Delta^r\}$ vanishes Q-linearly. {The same conclusion is true if the compactness assumption A(c) is replaced with the compactness of the level set $X^1:=\{x \mid f(x)\le f(x^1)\}$. }

\item For the R-BSUM algorithm, assume that the nonsmooth part has the form $h_k(x_k)=\lambda_k\|x_k\|_1$ for some $\lambda_k\ge 0$, and that $C_k$ has full row rank for each $k$. Then $\{\mathbb{E}[\Delta^t]\}$ vanishes Q-linearly.
\end{enumerate}
}
\end{theorem}
\begin{proof}
We first show part (1) of the claim.
{\color{black}
By directly adapting the proof of Lemma~\ref{lm:p-descent}-(1), we can show the following sufficient descent
\begin{align}
\Delta^r-\Delta^{r-1}\le -\gamma \|x^{r+1}-x^{r}\|^2,\nonumber
\end{align}
where $\gamma=\min_k\gamma_k$. This implies that $x^r\in X^1$ for all $r\ge 1$. By \eqref{eq:p-gap-bcd} in Lemma \ref{lm:cost-to-go}, we have that for all $r\ge 1$
\begin{align}
\Delta^r&\le \zeta\|x^{r+1}-x^r\|^2+\zeta^{'}\|x^r-\bar{x}^r\|^2\nonumber\\
&\le  \zeta\|x^{r+1}-x^r\|^2+\zeta^{'}\tau^2\|\tilde{\nabla} f(x^r)\|^2\nonumber\\
&\le  (\zeta+\zeta^{'}\tau^2\sigma^2)\|x^{r+1}-x^r\|^2\label{eq:bound_delta_by_difference}
\end{align}
where the last inequality is obtained by specializing Lemma \ref{lm:estimate}-(1) to the BSUM algorithm.}
Note that due to the compactness of either the feasible set $X$ or the level set $X^1$, the second inequality, which uses the error bound condition in Lemma \ref{lm:eb}, holds true for all $r\ge 1$.
Combining the previous two results, we have
\begin{align}
\Delta^r-\Delta^{r-1}\le -\frac{\gamma}{(\zeta+\zeta^{'}\tau^2\sigma^2)}\Delta^r,\quad \mbox{for all $r\ge 1$.}\nonumber
\end{align}
Define $\lambda:=\frac{\gamma}{(\zeta+\zeta^{'}\tau^2\sigma^2)}$, we conclude that $\Delta^r$ converges $Q$-linearly, that is,
\begin{align}
0\le \Delta^r\le\frac{1}{1+\lambda}\Delta^{r-1}\quad \mbox{for all $r\ge 1$}.\nonumber
\end{align}
%Additionally, if the sets $X_k$'s are not compact, then the error bound condition Lemma \ref{lm:eb} only holds when $x^r$ is sufficiently close to $X^*$. From Corollary \ref{cor:bcd}, we know that $x^r$ converges to the optimal $X^*$, which suggests that for any $\delta>0$, there must exist an index $\bar{r}$ such that for all $r>\bar{r}$, $\|\tilde{\nabla} f(x^r)\|\le\delta$, therefore the error bound condition holds true. In this case, \eqref{eq:bound_delta_by_difference} hence \eqref{eq:linear_bcd} are only true for all $r\ge \bar{r}$. As a result, we conclude that $\Delta^r\to 0$ R-linearly.

It remains to  show part (2) of the thereom. By adapting the proof of Lemma \ref{lm:p-descent}-(2) and taking full expectation, we have
\begin{align}
\mathbb{E}[\Delta^t-\Delta^{t-1}]\le -\hat{\gamma} \mathbb{E}[\|\hat{x}^{t+1}-x^{t}\|^2]\label{eq:p-descent-sbcd}.
\end{align}
where $\hat{\gamma}=\min_{k}p_k\gamma_k$,
%Taking full expectation, we have
%\begin{align}
%\mathbb{E}[\Delta^t-\Delta^{t-1}]\le -\hat{\gamma}\mathbb{E}[ \|\hat{x}^{t+1}-x^{t}\|^2].
%\end{align}
Using \eqref{eq:p-descent-sbcdmm} in Lemma \ref{lm:cost-to-go}, there exists a ${t}_0>0$ such that for all $t>{t}_0$, the following is true
\begin{align}\nonumber
\mathbb{E}[\Delta^t\mid x^{t}]&\le \hat{\zeta}\|x^t-\bar{x}^t\|^2\le  \hat{\zeta}\tau^2\|\tilde{\nabla} f(x^t)\|^2\le  \hat{\zeta}\tau^2\hat{\sigma}^2\|\hat{x}^{t+1}-x^t\|^2, \ \mbox{w.p.1.}, %\label{eq:bound_delta_by_difference_sbcd},
\end{align}
where the last inequality is obtained by specializing Lemma \ref{lm:estimate}-(2) to the R-BCD algorithm.
Taking full expectation, we obtain
\begin{align}\nonumber
\mathbb{E}[\Delta^t]\le \hat{\zeta}\tau^2\hat{\sigma}^2\mathbb{E}\left[\|\hat{x}^{t+1}-x^t\|^2\right],\ \forall~t\ge t_0.
\end{align}
Combining this with \eqref{eq:p-descent-sbcd} yields
\begin{align}\nonumber
\mathbb{E}\left[\Delta^t-\Delta^{t-1}\right]\le -\frac{\hat{\gamma}}{\hat{\zeta}\tau^2\hat{\sigma}^2}\mathbb{E}[\Delta^t], \ \forall~t\ge t_0.
\end{align}
Define $\hat{\lambda}:=\frac{\hat{\gamma}}{\hat{\zeta}\tau^2\hat{\sigma}^2}$, we conclude that there exists a $t_0>0$ such that
\begin{align}\nonumber
0\le \mathbb{E}[\Delta^t]\le\frac{1}{1+\hat{\lambda}}\mathbb{E}[\Delta^{t-1}], \ \forall~t\ge t_0.
\end{align}
implying that $\mathbb{E}[\Delta^t]$ vanishes Q-linearly.
\QED

\end{proof}

Recently the authors of \cite{Sanjabi13} have shown that the cyclic BCD algorithm converges R-linearly under assumptions similar to Assumption A, except that the compactness assumption (Assumption A(c)) is not required.
Compared with \cite{Sanjabi13}, the new elements in part (1) of Theorem \ref{thm:main_linear_bcd} are: (i) the cyclic BCD algorithm converges Q-linearly when the feasible set or the level set is compact; (ii) The same rate can be obtained when the per-block problem is minimized approximately by working with the approximate function $u_k(\cdot;\cdot)$.

\section{Numerical Results}
\vspace{-0.2cm}
In this section, we report numerical results that illustrate the effectiveness of the BSUM-M for large practical problems.
\subsection{Sovling a Linear System of Equations}
Recently, the authors of \cite{chen13} have demonstrated via a counterexample, that the classic two-block ADMM algorithm could diverge when applied to solve problems with three or more blocks. In the counterexample, ADMM is used to solve the following linear systems of equations (which has a unique solution $x_1=x_2=x_3=0$)
\begin{align}
&E_1 x_1 +E_2 x_2 + E_3 x_3 =0, \label{eq:Example}\\
&\mbox{with}\quad  [E_1 \; E_2\; E_3]=\left[
\begin{array}{lll}
1 & 1 & 1\\
1 & 1 & 2\\
1 & 2 & 2
\end{array}
\right].
\end{align}
It is shown in \cite{chen13} that regardless of the starting point, the ADMM algorithm always diverges. However,  we have shown in this paper that the BSUM-M is guaranteed to obtain the unique solution of the above linear system of equations\footnote{Since the $[E_1\; E_2\; E_3]$ is full rank, it follows that the augmented Lagrangian function is strongly convex and therefore the global error bound condition holds and the compactness assumption is not needed.}.
The following special version of the BSUM-M iteration for solving \eqref{eq:Example} has the same iteration as the ADMM except for a different dual stepsize.
\begin{align}
y^{r+1}   &= y^r+\alpha^r\left(E_1 x^r_1 +E_2 x^r_2 + E_3 x^r_3\right)\nonumber\\
x^{r+1}_1 &= (E^T_1E_1)^{-1}\left(-E^T_1 E_2 x^r_2 - E^T_1 E_3 x^r_3 - E^T_1 y^{r+1}/\rho\right)\nonumber\\
x^{r+1}_2 &= (E^T_2E_2)^{-1}\left(-E^T_2 E_1 x^{r+1}_1 - E^T_2 E_3 x^r_3 - E^T_2 y^{r+1}/\rho\right)\nonumber\\
x^{r+1}_3 &= (E^T_3E_3)^{-1}\left(-E^T_3 E_1 x^{r+1}_1 - E^T_3 E_2 x^{r+1}_2 - E^T_3 y^{r+1}/\rho\right)\nonumber.
\end{align}
In our experiment, we choose $\rho=1$ and $\alpha^r=\rho\times \frac{1}{\sqrt{r}}$. We run the BSUM-M and RBSUM-M for $1000$ trials, and for each trial we initialize the components in $x$ and $y$ uniformly randomly from $[-10, 10]$. For the RBSUM-M algorithm, the primal and dual blocks are picked with equal probability at each iteration. We see from Figs. \ref{figCounterExampleBCDMM}--\ref{figCounterExampleRBCDMM} below that in all the trials both algorithms converge nicely. The RBSUM-M takes longer time to converge, because at each iteration only a single primal or dual variable is updated.

\begin{figure}[htb]
        \begin{minipage}[t]{0.5\linewidth}
    \centering
    {\includegraphics[width=1\linewidth]{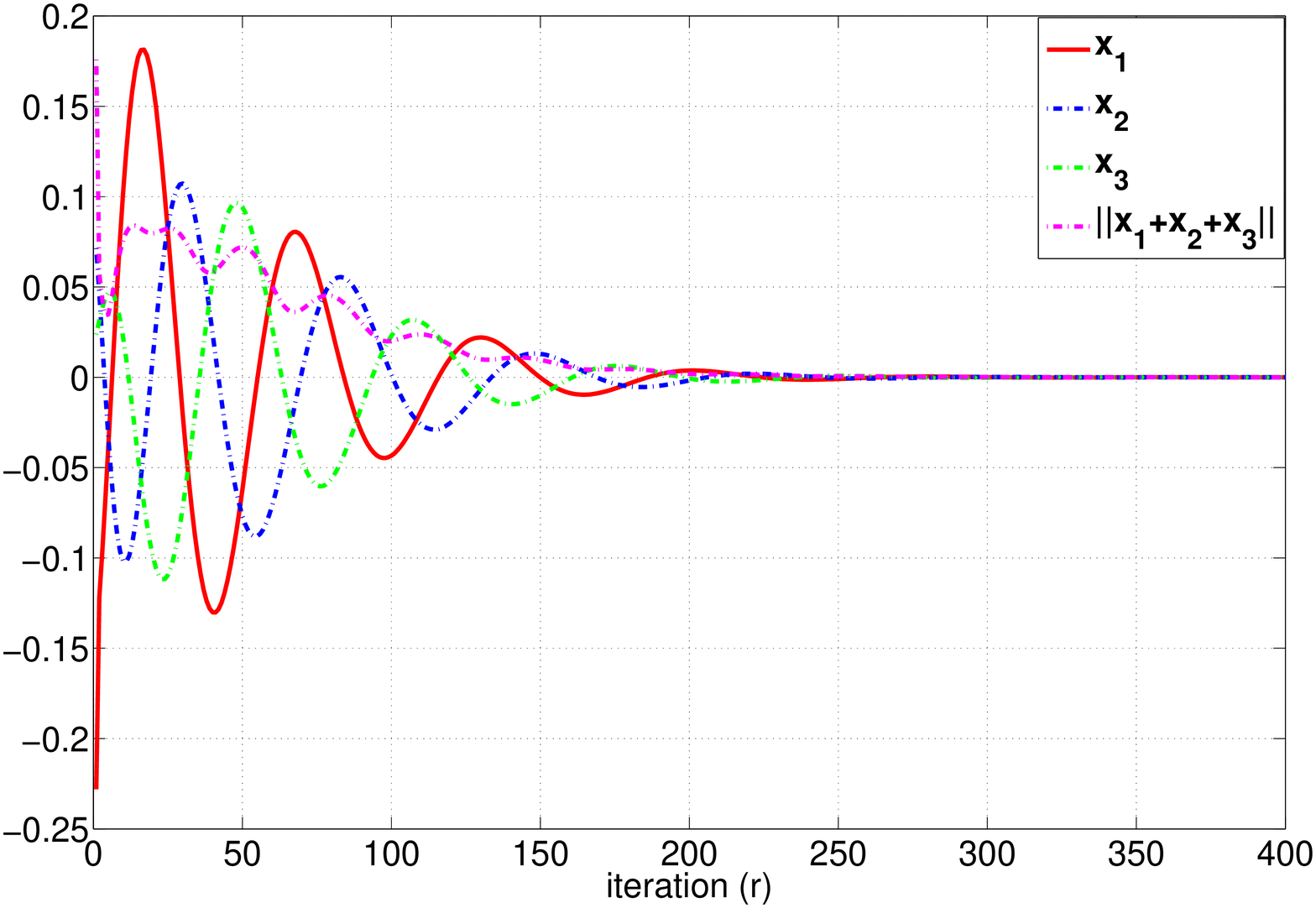}
\caption{Iterates generated by the BSUM-M   for solving \eqref{eq:Example}. Each curve is averaged over 1000 runs (with random starting points).}\label{figCounterExampleBCDMM}}
\end{minipage}\hfill
    \begin{minipage}[t]{0.5\linewidth}
    \centering
     {\includegraphics[width=1\linewidth]{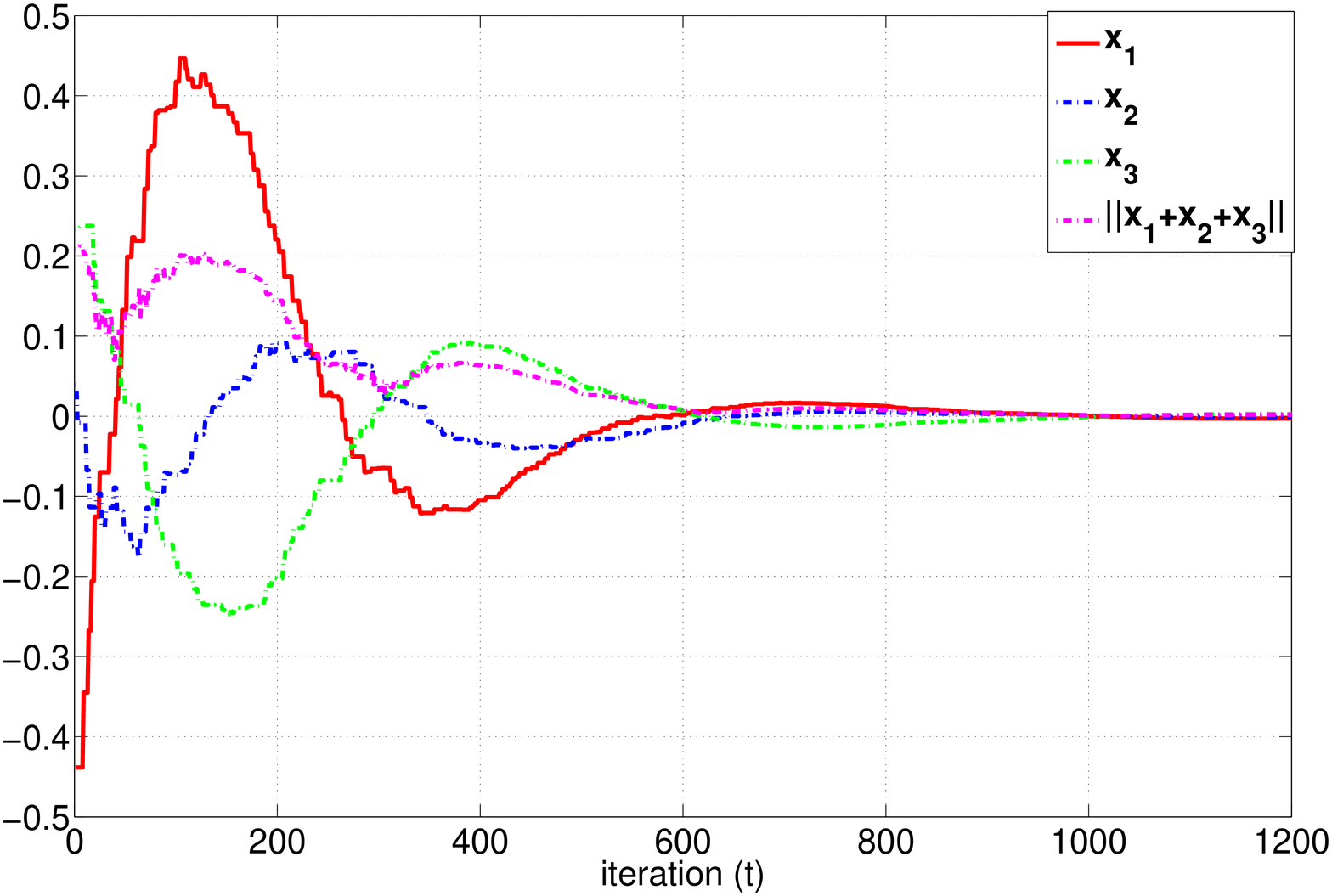}
\caption{Iterates generated by the RBSUM-M algorithm for solving \eqref{eq:Example}. Each curve is averaged over 1000 runs (with random starting points)}\label{figCounterExampleRBCDMM}}
\end{minipage}
 \end{figure}

\subsection{The BP Problem}
In the second experiment, we consider the BP problem \eqref{eq:BP_block}, and fix each block variable $x_k$ to be a {\it scalar}. Then the primal subproblem for the BSUM-M at the $r$-th iteration for the $k$-th variable is given by
\begin{align}
\min_{x_k}\frac{1}{\rho\|e_k\|^2}|x_k|+\frac{1}{2}\left(x_k+\frac{e^T_kc_k^r}{\|e_k\|^2}\right)^2
\end{align}
where $e_k$ is the $k$-th column of $E$, $c_k^r=w^r_{-k}+y^{r+1}/\rho-q$. This problem can be solved in closed-form by the soft-thresholding operator. It is worth noting that the update for each component variable $x_k$ only requires one piece of data $e_k$. Therefore in situations where the data matrix is only partially available at each update iteration \cite{nestrov12, richtarik12}, the randomized BSUM-M algorithm can be very valuable.

We randomly generate the matrix $E\in\Re^{m\times n}$ and the true solutions $\bar{x}$ with each of their nonzero components following standard Gaussian distribution. We let $E$ be a dense matrix, and $\bar{x}$ be a sparse vector, with each component having probability $p\in(0,1)$ to be nonzero (see \cite{Yang10L1} for details). We normalize the columns of $E$ to have norm 1. We have used the following stepsize rule for the BSUM-M  and the RBSUM-M: $\rho=10\times m/\|q\|_1$, $\alpha^r=\rho\frac{10+1}{\sqrt{r}+10}$. Unless specified explicitly, the blocks in the RBSUM-M are chosen uniformly with $p_k=\frac{1}{K+1}$ for all $k=0,\cdots, K$. The BSUM-M and the RBSUM-M are compared with a number of well-known algorithms for BP such as DALM, PALM \cite{Yang_alternatingdirection} and FISTA \cite{Beck:2009:FIS:1658360.1658364}; see \cite{Yang10L1} for a detailed review and implementation of these algorithms. In particular, for PALM, the primal and dual stepsizes are set equally to  $10\times m/\|q\|_1$; for DALM, the primal and dual stepsizes are set to $0.1\times |q\|_1/m$; for FISTA, backtrack line search is used (these are the default settings in the package \cite{Yang10L1}).

\begin{table}
\centering{\small\vspace{-0.1cm}
\begin{tabular}{|c|c|c|c|c|c|c|c|}
  \hline
  ${\boldsymbol{n}}$ & $\boldsymbol{m}$& $\boldsymbol{p}$ & {\bf BSUM-M} &{\bf RBSUM-M} & {\bf PALM} & {\bf DALM}& {\bf FISTA}\\
  \hline
  {\bf 10000} & {\bf 3000}& {\bf 0.06} & 226 &796 & 948& 840& 768\\
  \hline
  {\bf 10000} & {\bf 3000}& {\bf 0.01} & 74 &418 & 370 & 374& 584\\
  \hline
  {\bf 10000} & {\bf 5000}& {\bf 0.06} & 144 &670 & 542 & 604& 618\\
  \hline
  {\bf 10000} & {\bf 5000}& {\bf 0.01} & 64 &422 & 188 & 234& 582\\
  \hline
\end{tabular}}
\caption{\footnotesize Average $\#{\rm MVM}$ performance for different algorithms.}
\vspace{-0.3cm} \label{table:BPPerformance}
\end{table}

We first consider a relatively small problem. The stopping criteria for all the algorithms is that either the iteration counter is larger than $1000$, or the relative error $\|x^{r}-\bar{x}\|/\|\bar{x}\|\le 10^{-10}$. Fig. \ref{fig:BCDMM} shows the convergence behavior of all the algorithms for one instance of the problem with $n=10000$, $m=3000$ and $p=0.06$. For ease of exposition, in this figure each iteration of  the RBSUM-M consists of $10000$ random update steps. In Table \ref{table:BPPerformance}, we show the averaged performance (over $100$ problem realizations) for different algorithms. For a fair comparison of the computational cost, the algorithms are compared according to the number of matrix-vector multiplications, denoted by ${\rm \#MVM}$, which includes both $Ex$ and $E^T y$ (see e.g., \cite{Yang_alternatingdirection} for a similar definition). Clearly the BSUM-M approach exhibits superior performance over all other algorithms.

\begin{figure}
 \centering
 \vspace{-0.5cm}
\includegraphics[width=0.7\linewidth]{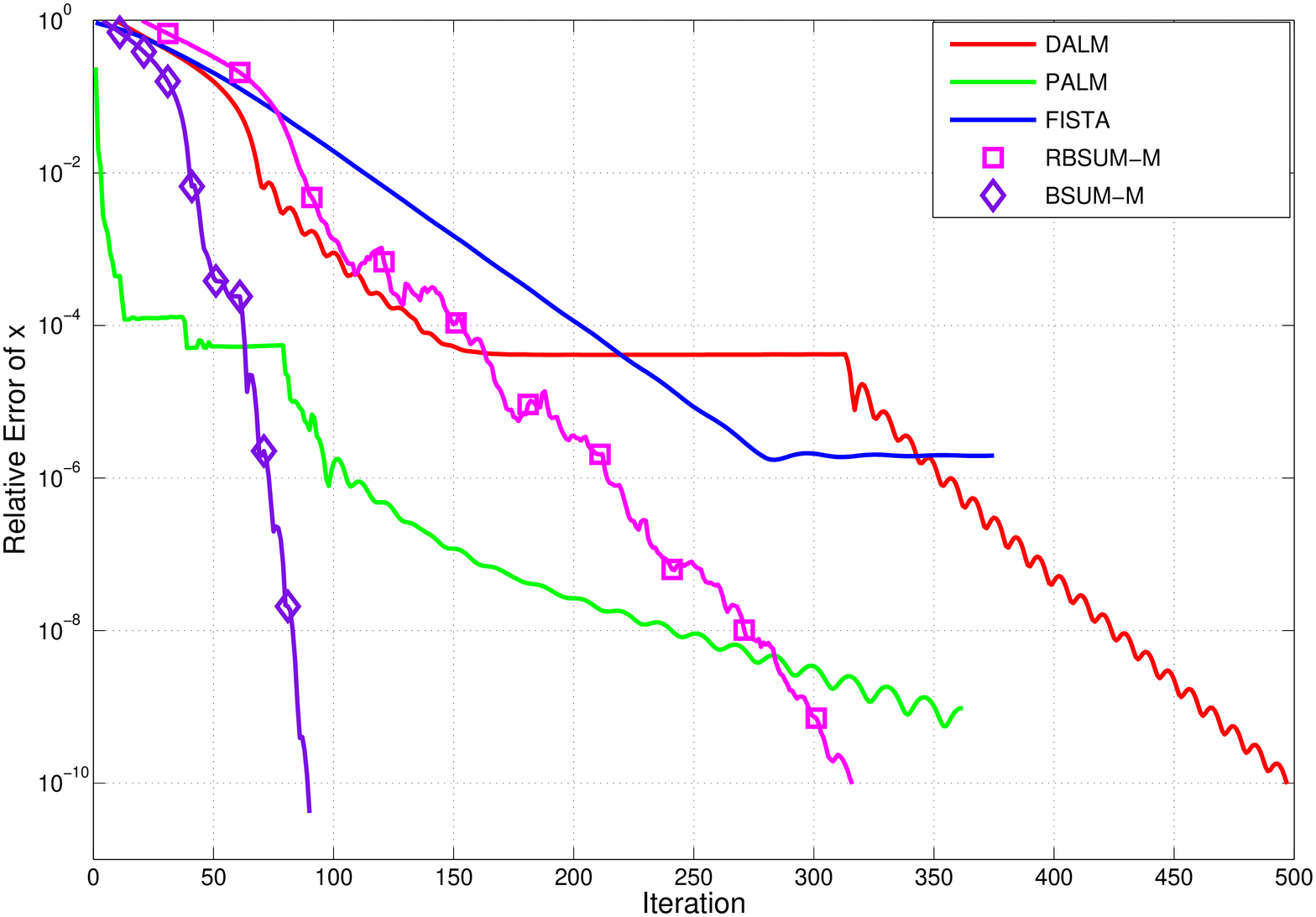}
\vspace{-0.4cm}
\caption{\footnotesize Relative error performance for all algorithms on a small-size problem. $n=10000$, $m=3000$, $p=0.06$. The relative error is given by $=\|\bar{x}-x^r\|/\|\bar{x}\|$.}\label{fig:BCDMM}
 \vspace{-0.1cm}
\end{figure}

It is worth mentioning that except for BSUM-M and RBSUM-M, all the rest of algorithms suffer from pitfalls that prevent them from solving really large problems. For example the PALM requires the knowledge of $\rho(E^{T}E)$ (the largest eigenvalue of $E^{T}E$), the version of DALM with convergence guarantee requires the inversion of $E E^T$ \cite{Yang_alternatingdirection}, both of which are difficult operations when $E$ is large (say when $n$ and $m$  are larger than $10^6$). The FISTA algorithm either needs $\rho(E^T E)$, or is required to perform backtrack line search within each iteration \cite{Beck:2009:FIS:1658360.1658364}, both of which are again difficult to implement for large size problems. In contrast, each step of the BSUM-M and RBSUM-M is simple and in closed-form, which makes it easily scalable for large problems. We have also tested the BSUM-M on two large experiments  \footnote{We use a PC with 128 GB RAM and 24 Intel Xeon 2.67 GHz cores.}: experiment 1 with $n=10^6$, $m=10^3$ and $\|\bar{x}\|_0=28$;  experiment 2 with $n=10^6$,  $m=2\times 10^3$ and $\|\bar{x}\|_0=82$. It takes  $7$ GB and $14$ GB of memory space to store the data of these problems, respectively. For both problems, the BSUM-M and RBSUM-M perform quite well: for the first (resp.\ the second) experiment they take around $15$ iterations and about $60$ seconds (resp. $25$ iterations and $200$ seconds) to reduce the relative error to about $10^{-6}$.

\begin{table}
\centering{\small\vspace{-0.1cm}
\begin{tabular}{|c|c|c|c|c|}
  \hline
  {\bf  \# of iterations $r$}& {\bf  Exp. 1} & {\bf  Exp. 2} \\
  \hline
  1 &1 &1 \\
  \hline
  5&0.35 & 0.35 \\
  \hline
  10&0.0012& 0.16 \\
  \hline
  15&7e-6 & 2e-3\\
  \hline
  20&N/A&1e-5 \\
  \hline
  25&N/A&8e-7 \\
  \hline
\end{tabular}}
\caption{\footnotesize Relative error performance of BSUM-M for large-scale problem.}
\label{table:BSUMM}
\end{table}

\begin{table}
\centering{\small\vspace{-0.1cm}
\begin{tabular}{|c|c|c|c|c|}
  \hline
  {\bf  \# of iterations $t$ ($\times 10^6$ )}& {\bf  Exp. 1} & {\bf  Exp. 2}\\
  \hline
  1 &1 &1 \\
  \hline
  5&0.05 & 0.18 \\
  \hline
  10&1e-4& 0.002 \\
  \hline
  15&2e-7 & 0.0019\\
  \hline
  20&N/A&0.0028 \\
  \hline
  25&N/A&6e-5 \\
    \hline
  30&N/A&9e-7 \\
  \hline
\end{tabular}}
\caption{\footnotesize Relative error performance of RBSUM-M for large-scale problem.}
\label{table:RBSUMM}
\end{table}

\subsection{The LASSO Problem}
In this section, we solve the LASSO problem
\begin{align}
\min_{x} \|Ax-b\|+\lambda\|x\|_1
\end{align}
using R-BCD and BCD, which are special cases of RBSUM-M and BSUM-M, respectively. For both algorithms each block variable again consists of a single scalar (i.e., $n_k$=1), so that no approximation is needed, and the per-block subproblem has a closed-form solution. Here our goal is not to establish the superiority of BCD-based algorithms in solving this type of problem (we refer the interested readers to \cite{scutari13flexible}, \cite{richtarik12} for comprehensive numerical studies for such purpose).  Rather, we wish to demonstrate that R-BCD may sometimes outperform the cyclic BCD and vice versa.

We use the instance generator proposed in \cite[Section 6]{nestrov07Gradient} to generate the problem data. After choosing the sparsity level for $A$ and $b$, the generator generates $A$, $b$, $x^*$.  We use $p_A$ (resp. $p_b$) to denote the probability for which each element of $A$ (resp. $b$) is nonzero. We also use the following formula to choose the update probability $p_k$ for each block $k$ \cite{nestrov12, richtarik12}
\begin{align}
p_k=\frac{L^{\alpha}_k}{\sum_{k}L^{\alpha}_k}, \ 0\le \alpha\le 1.
\end{align}

We first let $n=2000$, $m=1000$. The stopping criteria for both algorithms is that either the iteration counter is larger than $2000$, or the relative error $\|x^{r}-\bar{x}\|/\|\bar{x}\|\le 10^{-10}$. In Table \ref{table:LASSOPerformance}, we show the performance for R-BCD and BCD with different combinations of $p_b$ and $p_A$. Each entry in the table is an average of the results over $100$ realizations of the problem data. First we observe that using $0<\alpha\le 1$ improves the convergence significantly compared with uniform sampling (i.e., $\alpha=0$). Second, we see that R-BCD performs better when the data matrix is sparse ($p_A=0.01$), while the cyclic BCD outperforms R-BCD for the rest of the cases.

\begin{table}
\centering{\small\vspace{-0.1cm}
\begin{tabular}{|c|c|c|c|c|}
  \hline
  {\bf $\boldsymbol{p_b}$} & {\bf $\boldsymbol{ p_A}$} & {\bf R-BCD ($\alpha=0.5$)} &{\bf R-BCD ($\alpha=0$)} & {\bf BCD}\\
  \hline
  0.01 & 0.1 & 68 & 212 & 30\\
  \hline
  0.1 & 0.1 & 386 & 784 & 334\\
   \hline
  0.1 & 0.01 & 376 & 444 & 1242\\
     \hline
  0.05 & 0.01 & 180 & 294 & 529\\
  \hline
%  \hline
%    0.01 & 0.1 & 68 & 212 & 388 & 30\\
%  \hline
%  0.1 & 0.1 & 386 & 784 & 350 & 334\\
%   \hline
%  0.1 & 0.01 & 376 & 444 & 374 & 1242\\
%     \hline
%  0.05 & 0.01 & 180 & 294 & 374 & 529\\
%  \hline
\end{tabular}}
\caption{\footnotesize Average $\#{\rm MVM}$ performance for different algorithms, with $n=2000$, $m=1000$.}
\vspace{-0.3cm} \label{table:LASSOPerformance}
\end{table}

Next we consider the scenario where $n=\{50000, 30000, 10000\}$, $m=10000$,  $p_{A}=\{0.01, 0.001\}$ and $p_b=0.016$, and use the uniform sampling for the R-BCD algorithm. We plot the relative errors for both algorithms in Figs. \ref{figLASSO1}--\ref{figLASSO2}. In these figures we have again condensed $n$ random update steps for the R-BCD into a single iteration, so that the iteration numbers of the two algorithms are comparable. Comparing these two figures, we observe that both algorithms achieve better performance when the data matrix is sparser. In particular, the R-BCD converges faster than the cyclic BCD when $p_{A}=0.001$ and when $n/m$ is relatively small. Its performance degrades when $n/m$ becomes large. When $n=50000$, the R-BCD  does not show sign of convergence within the first few thousands of iterations. This observation was also noted in \cite[Section 6.14]{richtarik12} where the authors show that in a similar setting, it takes about $20000$ iterations (in each iteration all variables are updated once) for the R-BCD to converge to a reasonable solution. What is probably surprising here is that when $n=50000$, the cyclic BCD  performs quite well compared to the randomized version.

\begin{figure}[htb]
        \begin{minipage}[t]{0.5\linewidth}
    \centering
    {\includegraphics[width=1\linewidth]{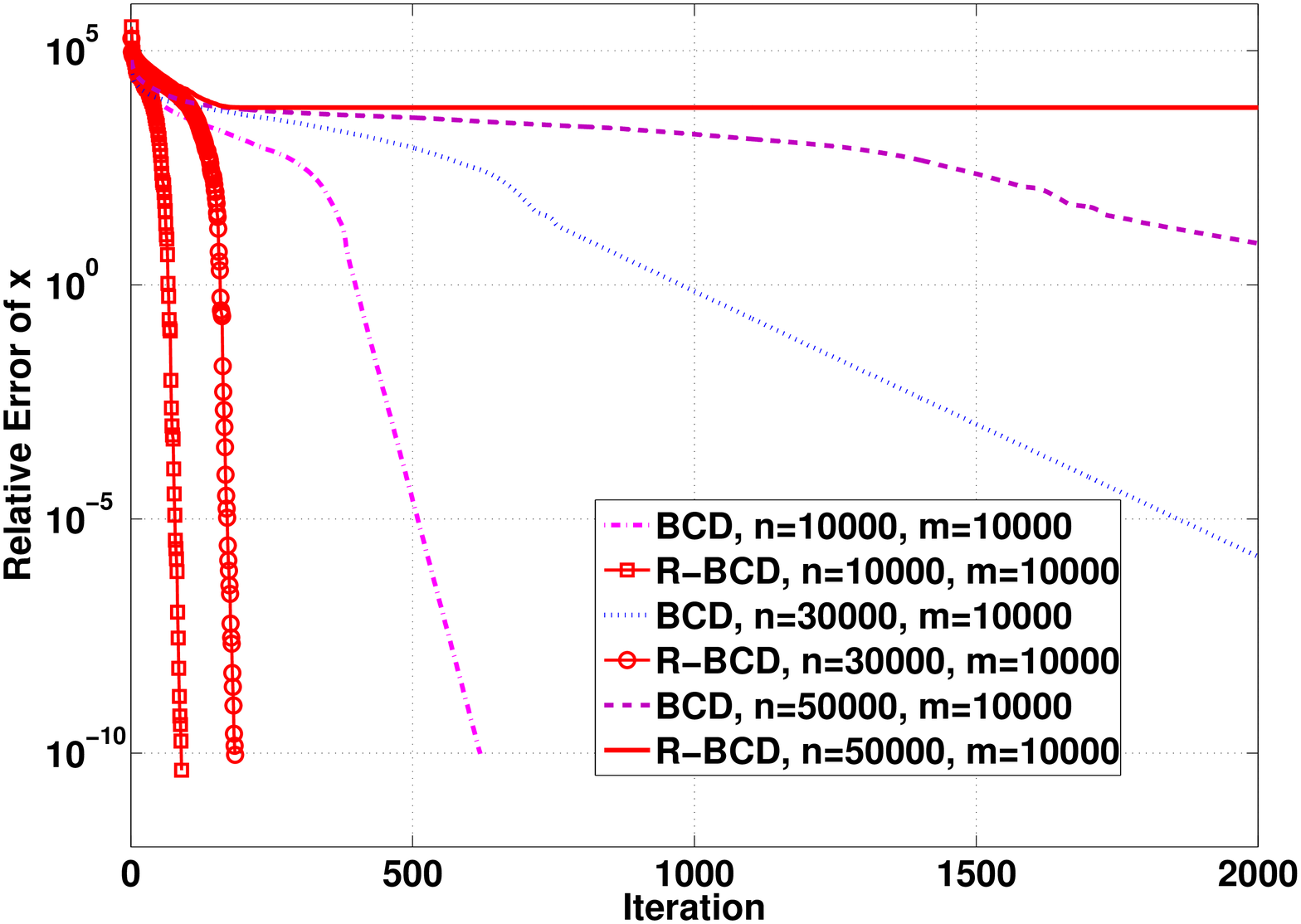}
\caption{Relative error performance for R-BCD and BCD . $p_A=0.001$, $p_b=0.016$, $n=\{10000, 30000, 50000\}$, $m=10000$.}\label{figLASSO1}}
\end{minipage}\hfill
    \begin{minipage}[t]{0.5\linewidth}
    \centering
     {\includegraphics[width=1\linewidth]{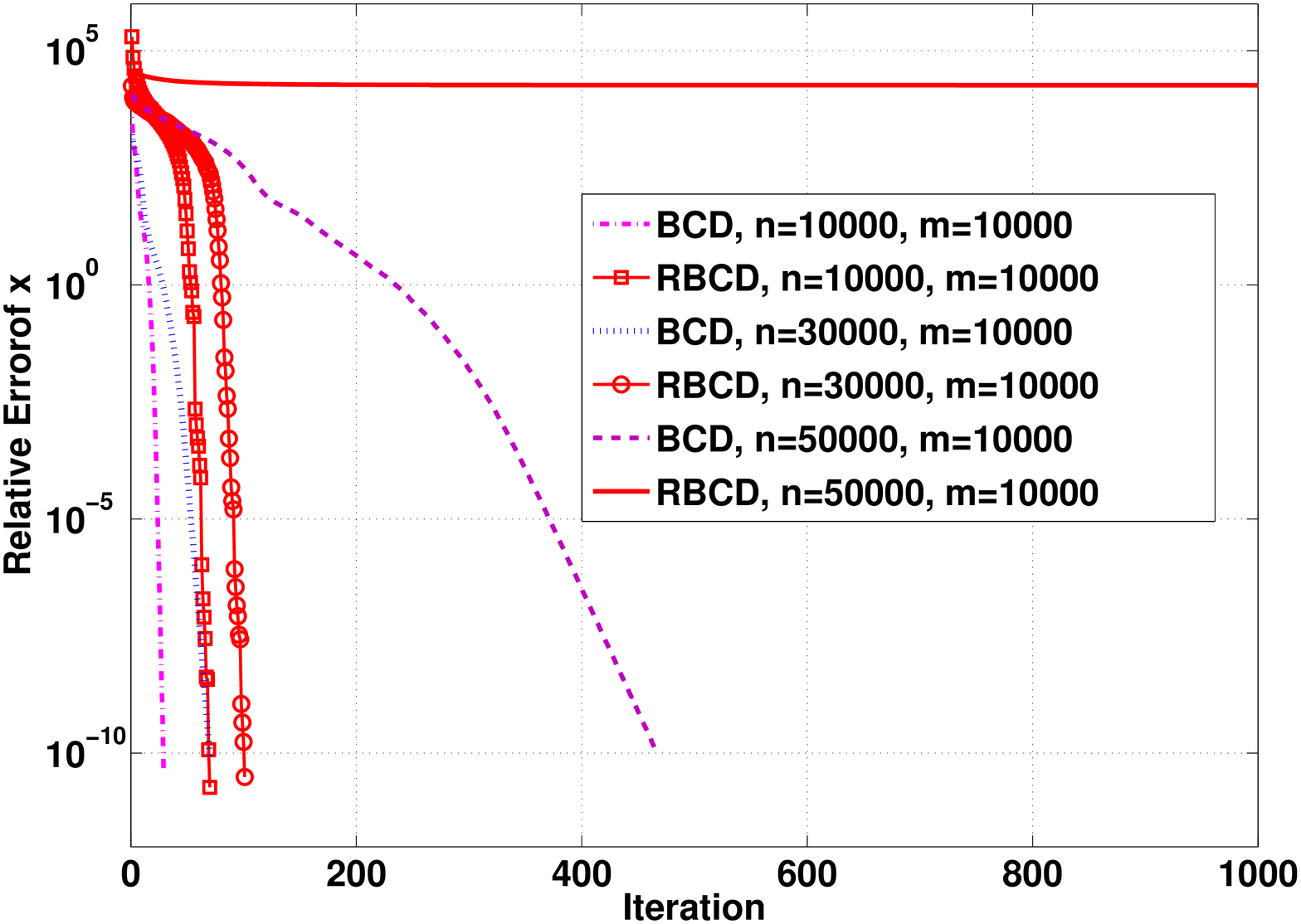}
\caption{Relative error performance for R-BCD and BCD . $p_A=0.01$, $p_b=0.016$, $n=\{10000, 30000, 50000\}$, $m=10000$.}\label{figLASSO2}}
\end{minipage}
 \end{figure}

\subsection{The DR Problem}
Let us now test the BSUM-M on the DR problem described in \eqref{eq:DREquiv}. Suppose that there are up to $3000$ users in the system with each user having $4$ controllable appliances; also assume that each day is divided into $96$ time periods. That is, $m=96$ and $n_k=96\times 4$. The load model is generated according to \cite{Paatero06}, and the detailed construction of the matrices $\{\boldsymbol{\Psi}_k\}_{k=1}^{K}$ can be found in \cite{Chang12}. For simplicity, we assume that the day-ahead bidding is completed, with power supply $\boldsymbol{p}$ determined by an average of $5$ random generation of all the uncontrolled consumptions of the users. This reduces problem \eqref{eq:DREquiv} to having only $\{\boldsymbol{x}_k\}_{k=1}^{K}$ and $\boldsymbol{z}$ as optimization variables. Additionally, we let $C_p(\cdot)$ and $C_s(\cdot)$ take the form of quadratic functions.

\begin{figure}[htb]
 \centering
 \vspace{-0.45cm}
\includegraphics[width=0.85\linewidth]{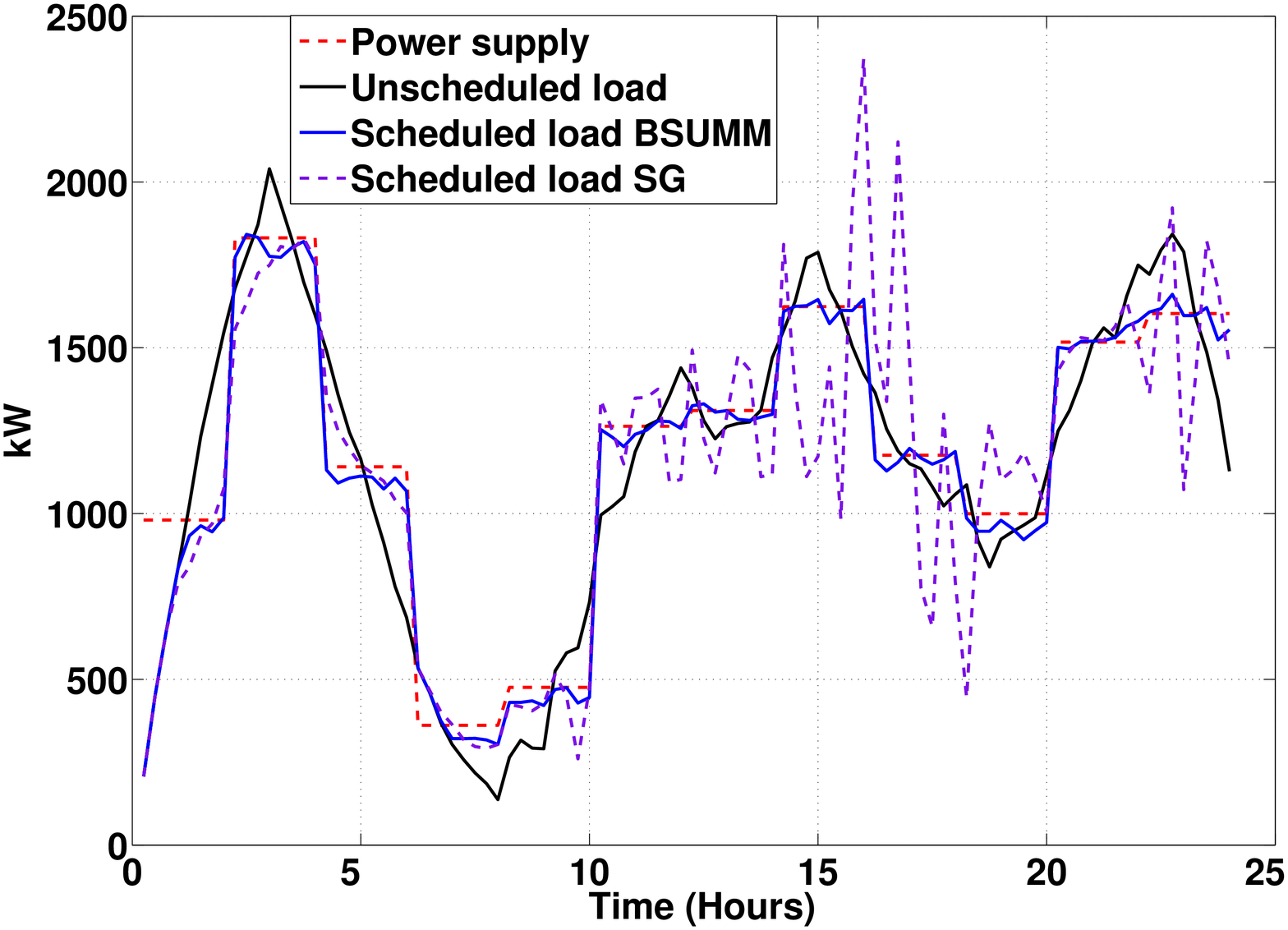}
 \vspace{-0.3cm}
\caption{\footnotesize The unscheduled consumption, power supply and the scheduled consumption by the BSUM-M and the subgradient algorithm.}\label{fig:DR}
 \vspace{-0.1cm}
\end{figure}

We compare the BSUM-M with the dual subgradient (SG) algorithm \cite{Chang12}\footnote{Note that here the dual SG is applied to the DR with quadratic costs, whereas the reference \cite{Chang12} employed {\it linear} costs.}. We let both algorithms run $200$ iterations. Note that each iteration of the SG is computationally more expensive, as it involves solving a linear program \cite{Chang12}, while each iteration of the BSUM-M is again in closed-form. In Table \ref{table:DR}, we compare the total costs of the scheduled loading solutions generated by the BSUM-M and the SG with that of unscheduled loads. Clearly the BSUM-M is able to achieve about $50\%$ of cost reduction, while the SG algorithm fails to converge within $200$ iterations which results in significantly larger costs. In Fig.\ \ref{fig:DR}, we plot the power supply, the consumption levels of unscheduled loads as well as those scheduled by the BSUM-M and the SG. We can see that the BSUM-M can track the supply curve quite well, while the SG fails to do so within $200$ iterations.

\begin{table}
\centering{\small\vspace{-0.1cm}
\begin{tabular}{|c|c|c|c|c|c|}
  \hline
  {\bf Algorithm}& $\boldsymbol{K=50}$& $\boldsymbol {K=100}$& $\boldsymbol {K=500}$& $\boldsymbol {K=1000} $& $\boldsymbol {K=3000}$\\
  \hline
  {\bf BSUM-M} & 0.4860 & 0.8099 & 3.3964 &4.648 & 14.827\\
  \hline
  {\bf SG} & 0.9519 & 1.5630 & 9.4835 &16.595 & 60.896\\
  \hline
  {\bf Unscheduled} & 1.0404 & 1.7940 & 7.5749 &14.389 & 45.900\\
  \hline
\end{tabular}}
\caption{\footnotesize Total Cost Performance of Different Approaches ($10^3$ unit price).}
\vspace{-0.5cm} \label{table:DR}
\end{table}

\section{Conclusion}
In this paper, we propose a first order primal-dual method for nonsmooth convex minimization problems subject to linear constraints. The new algorithm, which we call the block successive upper-bound minimization method of multipliers (BSUM-M), alternates between simple primal and dual steps either randomly or deterministically, and is well suited for large scale applications involving big data. In the primal steps, certain locally tight upper-bounds of the augmented Lagrangian function are successively minimized, while the dual step is in closed form and follows an approximate dual ascent step. The algorithm is a generalization of the ADMM method and the BCD method in that it offers greater flexibility both in choosing a suitable upper-bound function in place of the augmented Lagrangian function when performing the primal update, and in the order of primal-dual updates. We have established the convergence of the BSUM-M algorithm (for both the deterministic and randomized versions) and have demonstrated their strong numerical performance for large scale realistic applications. In future, it will be interesting to study if the BSUM-M can converge to a local stationary point for nonconvex problems, and if so, how effective it is in practical applications.

\bibliographystyle{IEEEbib}

\bibliography{ref,biblio}
\end{document}